\documentclass[a4paper,twoside]{article}
\usepackage{a4}
\usepackage{amssymb}
\usepackage{amsmath}
\usepackage{upref}
\usepackage[active]{srcltx}
\usepackage[pagebackref,colorlinks,citecolor=blue,linkcolor=blue,urlcolor=blue]{hyperref}
\usepackage[dvipsnames]{color}
\allowdisplaybreaks[2] 
\newcount\minutes \newcount\hours
\hours=\time
\divide\hours 60
\minutes=\hours
\multiply\minutes -60
\advance\minutes \time
\newcommand{\klockan}{\the\hours:{\ifnum\minutes<10 0\fi}\the\minutes}
\newcommand{\tid}{\today\ \klockan}
\newcommand{\prtid}{\smash{\raise 10mm \hbox{\LaTeX ed \tid}}}
\renewcommand{\prtid}{}
\makeatletter
\def\sectionmark#1{} 
\def\subsectionmark#1{}
\newcommand{\sectnr}{\ifnum \c@secnumdepth >\z@
                 \thesection.\hskip 1em\relax \fi}
\def\@evenhead{\footnotesize\rm\thepage\hfil\leftmark\hfil\llap{\prtid}}
\def\@oddhead{\footnotesize\rm\rlap{\prtid}\hfil\rightmark\hfil\thepage}
\def\tableofcontents{\section*{Contents} 
 \@starttoc{toc}}
\makeatother
\makeatletter
\def\@biblabel#1{#1.}
\makeatother
\makeatletter
\let\Thebibliography=\thebibliography
\renewcommand{\thebibliography}[1]{\def\@mkboth##1##2{}\Thebibliography{#1}
\addcontentsline{toc}{section}{References}
\frenchspacing 
\setlength{\@topsep}{0pt}
\setlength{\itemsep}{0pt}%
\setlength{\parskip}{0pt plus 2pt}%
}
\makeatother
\makeatletter
\def\mdots@{\mathinner.\nonscript\!.%
 \ifx\next,.\else\ifx\next;.\else\ifx\next..\else
 \nonscript\!\mathinner.\fi\fi\fi}
\let\ldots\mdots@
\let\cdots\mdots@
\let\dotso\mdots@
\let\dotsb\mdots@
\let\dotsm\mdots@
\let\dotsc\mdots@
\def\vdots{\vbox{\baselineskip2.8\p@ \lineskiplimit\z@
    \kern6\p@\hbox{.}\hbox{.}\hbox{.}\kern3\p@}}
\def\ddots{\mathinner{\mkern1mu\raise8.6\p@\vbox{\kern7\p@\hbox{.}}%
    \raise5.8\p@\hbox{.}\raise3\p@\hbox{.}\mkern1mu}}
\makeatother
\makeatletter
\let\Enumerate=\enumerate
\renewcommand{\enumerate}{\Enumerate%
\setlength{\@topsep}{0pt}
\setlength{\itemsep}{0pt}%
\setlength{\parskip}{0pt plus 1pt}%
\renewcommand{\theenumi}{\textup{(\alph{enumi})}}%
\renewcommand{\labelenumi}{\theenumi}%
}
\let\endEnumerate=\endenumerate
\renewcommand{\endenumerate}{\endEnumerate\unskip}
\makeatother
\makeatletter
\def\@seccntformat#1{\csname the#1\endcsname.\quad}
\makeatother
\newcommand{\authortitle}[2]{\author{#1}\title{#2}\markboth{#1}{#2}}
\newcommand{\auth}[2]{{#1, #2.}}
\newcommand{\art}[6]{{\sc #1, \rm #2, \it #3\/ \bf #4 \rm (#5), \mbox{#6}.}}
\newcommand{\artin}[3]{{\sc #1, \rm #2, in #3.}}
\newcommand{\artprep}[3]{{\sc #1, \rm #2, \rm #3.}}

\newcommand{\book}[3]{{\sc #1, \it #2, \rm #3.}}
\newcommand{\AND}{{\rm and }}
\RequirePackage{amsthm}
\newtheoremstyle{descriptive}%
  {\topsep}   
  {\topsep}   
  {\rmfamily} 
  {}          
  {\bfseries} 
  {.}         
  { }         
  {}          
\newtheoremstyle{propositional}%
  {\topsep}   
  {\topsep}   
  {\itshape}  
  {}          
  {\bfseries} 
  {.}         
  { }         
  {}          
\theoremstyle{propositional}
\newtheorem{thm}{Theorem}[section]
\newtheorem{prop}[thm]{Proposition}
\newtheorem{lem}[thm]{Lemma}
\newtheorem{cor}[thm]{Corollary}
\theoremstyle{descriptive}
\newtheorem{deff}[thm]{Definition}
\newtheorem{example}[thm]{Example}
\newtheorem{remark}[thm]{Remark}
\makeatletter
\renewenvironment{proof}[1][\proofname]{\par
  \pushQED{\qed}%
  \normalfont
  \trivlist
  \item[\hskip\labelsep
        \itshape
    #1\@addpunct{.}]\ignorespaces
}{%
  \popQED\endtrivlist\@endpefalse
}
\makeatother
\makeatletter
\newdimen\pl@left
\newdimen\pl@down
\newdimen\pl@right
\newdimen\pl@temp
\def\sob#1#2#3#4#5{
  \setbox0\hbox{#1}\setbox1\hbox{$_\mathchar'454$}\setbox2\hbox{p}%
  \pl@right=#2\wd0 \advance\pl@right by-#3\wd1
  \pl@down=#5\ht1 \advance\pl@down by-#4\ht0
  \pl@left=\pl@right \advance\pl@left by\wd1
  \pl@temp=-\pl@down \advance\pl@temp by\dp2 \dp1=\pl@temp
  \leavevmode
  \kern\pl@right\lower\pl@down\box1\kern-\pl@left #1}
\def\aob{\sob a{.66}{.20}{0}{.90}}

\makeatother
\newcommand{\setm}{\setminus}
\renewcommand{\subsetneq}{\varsubsetneq}
\renewcommand{\emptyset}{\varnothing}
\def\vint{\mathop{\mathchoice%
          {\setbox0\hbox{$\displaystyle\intop$}\kern 0.22\wd0%
           \vcenter{\hrule width 0.6\wd0}\kern -0.82\wd0}%
          {\setbox0\hbox{$\textstyle\intop$}\kern 0.2\wd0%
           \vcenter{\hrule width 0.6\wd0}\kern -0.8\wd0}%
          {\setbox0\hbox{$\scriptstyle\intop$}\kern 0.2\wd0%
           \vcenter{\hrule width 0.6\wd0}\kern -0.8\wd0}%
          {\setbox0\hbox{$\scriptscriptstyle\intop$}\kern 0.2\wd0%
           \vcenter{\hrule width 0.6\wd0}\kern -0.8\wd0}}%
          \mathopen{}\int}
\newcommand{\ConeX}{{C_1^X}}
\newcommand{\ConeXhat}{{C_1^{\Xhat}}}
\newcommand{\ConeXt}{{C_1^{\Xt}}}
\newcommand{\ConeGx}{{C_1^{G_x}}}
\newcommand{\ConeOm}{{C_1^\Om}}
\newcommand{\CpG}{{C_p^G}}
\newcommand{\CpRtwo}{{C_p^{\R^2}}}
\newcommand{\CpRn}{{C_p^{\R^n}}}
\newcommand{\CpclGj}{{\displaystyle C_p^{\clGj}}}
\DeclareMathOperator{\capp}{cap}
\newcommand{\cpX}{\capp_p^X}
\newcommand{\cpclGj}{\mathop{\displaystyle \capp_p^{\clG_j}}}
\newcommand{\coneY}{\capp_1^Y}
\newcommand{\cpRn}{\capp_p^{\R^n}}
\DeclareMathOperator{\dist}{dist}

\newcommand{\bdy}{\partial}
\newcommand{\loc}{_{\rm loc}}
\DeclareMathOperator{\Mod}{Mod}
\newcommand{\Modp}{{\Mod_p}}
\DeclareMathOperator{\ACL}{ACL}
\DeclareMathOperator{\length}{length}
\DeclareMathOperator{\fineint}{fine-int}
\DeclareMathOperator{\interior}{int}
{\catcode`p =12 \catcode`t =12 \gdef\eeaa#1pt{#1}}      
\def\accentadjtext#1{\setbox0\hbox{$#1$}\kern   
                \expandafter\eeaa\the\fontdimen1\textfont1 \ht0 }
\def\accentadjscript#1{\setbox0\hbox{$#1$}\kern 
                \expandafter\eeaa\the\fontdimen1\scriptfont1 \ht0 }
\def\accentadjscriptscript#1{\setbox0\hbox{$#1$}\kern   
                \expandafter\eeaa\the\fontdimen1\scriptscriptfont1 \ht0 }
\def\accentadjtextback#1{\setbox0\hbox{$#1$}\kern       
                -\expandafter\eeaa\the\fontdimen1\textfont1 \ht0 }
\def\accentadjscriptback#1{\setbox0\hbox{$#1$}\kern     
                -\expandafter\eeaa\the\fontdimen1\scriptfont1 \ht0 }
\def\accentadjscriptscriptback#1{\setbox0\hbox{$#1$}\kern 
                -\expandafter\eeaa\the\fontdimen1\scriptscriptfont1 \ht0 }
\def\itoverline#1{{\mathsurround0pt\mathchoice
        {\rlap{$\accentadjtext{\displaystyle #1}
                \accentadjtext{\vrule height1.593pt}
                \overline{\phantom{\displaystyle #1}
                \accentadjtextback{\displaystyle #1}}$}{#1}}
        {\rlap{$\accentadjtext{\textstyle #1}
                \accentadjtext{\vrule height1.593pt}
                \overline{\phantom{\textstyle #1}
                \accentadjtextback{\textstyle #1}}$}{#1}}
        {\rlap{$\accentadjscript{\scriptstyle #1}
                \accentadjscript{\vrule height1.593pt}
                \overline{\phantom{\scriptstyle #1}
                \accentadjscriptback{\scriptstyle #1}}$}{#1}}
        {\rlap{$\accentadjscriptscript{\scriptscriptstyle #1}
                \accentadjscriptscript{\vrule height1.593pt}
                \overline{\phantom{\scriptscriptstyle #1}
                \accentadjscriptscriptback{\scriptscriptstyle #1}}$}{#1}}}}
\newdimen\extrawidth
\def\iintlim#1#2{\setbox0\hbox{$\scriptstyle#1$}%
	\setbox1\hbox{$\scriptstyle#2$}%
	\extrawidth=\wd1 \advance\extrawidth-\wd0
	\ifdim\extrawidth<0pt \extrawidth=0pt\fi%
	\int_{#1\kern\extrawidth \kern .5em}^{#2\kern -\wd1} \kern -.5em%
}
\newcommand{\al}{\alpha}
\newcommand{\be}{\beta}
\newcommand{\ga}{\gamma}
\newcommand{\gah}{\widehat{\gamma}}
\newcommand{\eps}{\varepsilon}
\newcommand{\la}{\lambda}
\newcommand{\de}{\delta}
\newcommand{\Om}{\Omega}
\renewcommand{\phi}{\varphi}
\newcommand{\p}{{$p\mspace{1mu}$}}
\newcommand{\R}{\mathbf{R}}
\newcommand{\eR}{{\overline{\R}}}
\newcommand{\Q}{\mathbf{Q}}
\newcommand{\limplus}{{\mathchoice{\vcenter{\hbox{$\scriptstyle +$}}}
  {\vcenter{\hbox{$\scriptstyle +$}}}
  {\vcenter{\hbox{$\scriptscriptstyle +$}}}
  {\vcenter{\hbox{$\scriptscriptstyle +$}}}
}}
\newcommand{\limminus}{{\mathchoice{\vcenter{\hbox{$\scriptstyle -$}}}
  {\vcenter{\hbox{$\scriptstyle -$}}}
  {\vcenter{\hbox{$\scriptscriptstyle -$}}}
  {\vcenter{\hbox{$\scriptscriptstyle -$}}}
}}
\newcommand{\None}{N^{1,1}}
\newcommand{\Noneloc}{N^{1,1}\loc}
\newcommand{\hNp}{\widehat{N}^{1,p}}
\newcommand{\CpXhat}{{C_p^{\Xhat}}}
\newcommand{\Dp}{D^p}
\newcommand{\Nploc}{N^{1,p}\loc}
\newcommand{\Dploc}{D^{p}\loc}
\newcommand{\Np}{N^{1,p}}
\newcommand{\CpX}{{C_p^X}}
\newcommand{\CpY}{{C_p^Y}}
\newcommand{\ConeY}{{C_1^Y}}
\newcommand{\hDp}{\widehat{D}^{p}}
\newcommand{\Done}{D^1}
\newcommand{\Doneloc}{D^{1}\loc}
\newcommand{\Lploc}{L^p\loc}
\newcommand{\Ga}{\Gamma}
\newcommand{\Xhat}{{\widehat{X}}}
\newcommand{\Yhat}{{\widehat{Y}}}
\newcommand{\Omhat}{\Om^\wedge}
\newcommand{\Bhat}{{\widehat{B}}}
\newcommand{\Ghat}{{\widehat{G}}}
\newcommand{\ghat}{{\hat{g}}}
\newcommand{\gt}{{\tilde{g}}}
\newcommand{\uhat}{{\hat{u}}}
\newcommand{\ut}{{\tilde{u}}}
\newcommand{\vt}{{\tilde{v}}}
\newcommand{\clG}{\itoverline{G}}
\newcommand{\clGx}{\itoverline{G}_x}
\newcommand{\clGj}{\itoverline{G}_{j}}
\newcommand{\Bik}{B_{ik}}
\newcommand{\rik}{r_{ik}}
\newcommand{\phiik}{\phi_{ik}}
\newcommand{\phijk}{\phi_{jk}}
\newcommand{\Bjk}{B_{jk}}
\newcommand{\rjk}{r_{jk}}
\newcommand{\CPI}{C_{\rm PI}}
\newcommand{\Cmu}{C_{\mu}}
\newcommand{\Leb}{\mathcal{L}}
\newcommand{\simae}{\stackrel{\textup{ae}}{\sim}}
\newcommand{\Hone}{\mathcal{H}^1}
\newcommand{\Hnone}{\mathcal{H}^{n-1}}
\newcommand{\dinK}{d_{{\rm in},\itoverline{\La B}_1}}
\newcommand{\Bin}{B_{\rm in}}
\newcommand{\La}{\Lambda}
\newcommand{\muhat}{\hat{\mu}}
\newcommand{\mut}{\tilde{\mu}}
\newcommand{\rhot}{\tilde{\rho}}
\newcommand{\muhatin}{\hat{\mu}_{\rm in}}
\newcommand{\muX}{\mu_{X}}
\newcommand{\muY}{\mu_{Y}}
\newcommand{\muYin}{\mu_{Y,\rm in}}
\newcommand{\clE}{\itoverline{E}}
\newcommand{\Et}{\widetilde{E}}
\newcommand{\Xt}{\widetilde{X}}
\numberwithin{equation}{section}
\newcommand{\imp}{\mathchoice{\quad \Longrightarrow \quad}{\Rightarrow}
                {\Rightarrow}{\Rightarrow}}
\newcommand{\eqv}{\mathchoice{\quad \Longleftrightarrow \quad}{\Leftrightarrow}
                {\Leftrightarrow}{\Leftrightarrow}}
\newcommand{\revimp}{\mathchoice{\quad \Longleftarrow \quad}{\Leftarrow}
                {\Leftarrow}{\Leftarrow}}
\newenvironment{ack}{\medskip{\it Acknowledgement.}}{}

\begin{document}

\authortitle{Anders Bj\"orn, Jana Bj\"orn and Panu Lahti}
{Removable sets for Newtonian Sobolev spaces $\ldots$}
\title{Removable sets for Newtonian Sobolev spaces and a characterization
of \p-path almost open sets}

\author{
Anders Bj\"orn \\
\it\small Department of Mathematics, Link\"oping University, SE-581 83 Link\"oping, Sweden\\
\it \small anders.bjorn@liu.se, ORCID\/\textup{:} 0000-0002-9677-8321
\\
\\
Jana Bj\"orn \\
\it\small Department of Mathematics, Link\"oping University, SE-581 83 Link\"oping, Sweden\\
\it \small jana.bjorn@liu.se, ORCID\/\textup{:} 0000-0002-1238-6751
\\
\\
Panu Lahti \\
\it\small Academy of Mathematics and Systems Science, Chinese Academy of Sciences, \\
\it\small Beijing 100190, PR China\/{\rm ;}
\it \small panulahti@amss.ac.cn, ORCID\/\textup{:} 0000-0002-1058-1625
}

\date{Preliminary version, \today}
\date{{To appear in} \emph{Rev. Mat. Iberoam.}}
\maketitle

\noindent{\small {\bf Abstract}.
We study removable sets for Newtonian Sobolev functions in metric measure spaces
  satisfying the usual (local) assumptions of a doubling measure and a Poincar\'e inequality.
In particular, when restricted to Euclidean spaces,
 a closed set $E \subset  \R^n$ with zero Lebesgue measure
is shown to be removable for $W^{1,p}(\R^n \setm E)$ if and only if $\R^n \setm E$
supports
a \p-Poincar\'e inequality as a metric space.
When $p>1$,  this recovers Koskela's result
(\emph{Ark.\ Mat.\ }{\bf 37} (1999), {291--304}),
but for $p=1$, as well as for metric spaces, it seems to be new.
We also obtain the corresponding characterization for the Dirichlet
  spaces $L^{1,p}$.
To be able to include $p=1$,
we first study extensions of Newtonian Sobolev functions in the case $p=1$
	from a noncomplete space $X$ to its completion $\Xhat$.

In these results, \p-path almost open sets play an important role,
and we  provide a characterization of them by means of \p-path
open, \p-quasiopen and \p-finely open sets.
We also show that there are nonmeasurable \p-path almost open subsets
of $\R^n$, $n \ge 2$, provided that the continuum hypothesis is assumed to be true.

Furthermore, we extend earlier results about
measurability of functions with $L^p$-integrable upper gradients,
about \p-quasiopen, \p-path and \p-finely open sets,
and about Lebesgue points for $\None$-functions,
to spaces that only satisfy local assumptions.

\medskip
\noindent
{\small \emph{Key words and phrases}:
Dirichlet space,
Lebesgue point,
local doubling measure,
noncomplete metric space, 
Newtonian space,
Poincar\'e inequality,
\p-finely open set,
\p-path almost open set,
\p-path open set,
removable set,
Sobolev space.
}

\medskip
\noindent
{\small Mathematics Subject Classification (2020):
Primary:
31E05; 
Secondary:  
26D10, 
30L15, 
30L99, 
31C45, 
46E35, 
46E36.} 
}

\section{Introduction}

The recent development in analysis on metric spaces has made it possible
to define Sobolev type spaces also on nonopen subsets of $\R^n$. 
This, in particular, leads to questions about extensions and
restrictions of Sobolev functions, as well as about 
the gradients of such restrictions in arbitrary (possibly nonmeasurable) sets.
In this paper, we address some of these questions in rather general
metric spaces and sets.

Standard assumptions in the area 
are that the metric space
is complete and equipped with a globally doubling measure supporting a
global \p-Poincar\'e inequality. 
The integrability exponent for Sobolev functions and their
gradients is often assumed to be $p>1$, since this gives
reflexive spaces and provides useful tools. 
At the same time, in many concrete situations, it is desirable to
consider noncomplete spaces and to relax the global assumptions to
local ones.
Last, but not least, the case $p=1$ is also attracting a lot of interest.

It was shown by Koskela~\cite[Theorem~C]{Koskela} that a
closed set
$E\subset\R^n$ of zero Lebesgue measure is removable for the Sobolev
space $W^{1,p}(\R^n\setm E)$, with $p>1$, if and only if 
$\R^n\setm E$ supports a \p-Poincar\'e inequality.
One of our results is that a similar equivalence holds also for $p=1$
and for metric spaces, even noncomplete ones and with only local Poincar\'e inequalities.
Moreover, we do not require $E$ to be closed, only that its
complement $\Om$ is \p-path almost open, i.e.\ 
for \p-almost every curve $\ga$, 
the preimage $\ga^{-1}(\Om)$ is a union of an open set and a set of zero 
1-dimensional Lebesgue measure.

When specialized to weighted Euclidean spaces, as in
Heinonen--Kilpel\"ainen--Martio~\cite{HeKiMa}, these results (obtained 
in Theorems~\ref{thm-hN1-rem-global} and~\ref{thm-hN1-rem-local}) can
be formulated as follows. 
Here we follow the notation of~\cite{HeKiMa} and denote
the weighted Sobolev and Dirichlet spaces by $H^{1,p}(\Om,\mu)$ and
$L^{1,p}(\Om,\mu)$, respectively. 
These spaces coincide with $\hNp(\Om)$ and $\hDp(\Om)$,
with respect to $\mu$, as defined in Section~\ref{sect-rem},
see the discussion after Theorem~\ref{thm-hN1-rem-global}.

\begin{thm} \label{thm-A-intro}
Let $1 \le p <\infty$ and $d\mu=w\,dx$, where $w$ is a
\p-admissible weight on $\R^n$ in the sense of~\textup{\cite{HeKiMa}}. 
Let $\Om=\R^n \setm E$, where $\mu(E)=0$.
Assume that $\Om$ is \p-path almost open,
which in particular holds if $\Om$ is open.

Then the following statements are equivalent\/\textup{:}
\begin{enumerate}
\item \label{A-None}
$E$ is removable for the Sobolev space $H^{1,p}(\Om,\mu)$.
\item 
$E$ is removable for the Dirichlet space $L^{1,p}(\Om,\mu)$.
\item \label{A-X}
$(\Om,\mu)$ supports a
  global \p-Poincar\'e inequality.
\end{enumerate}
\end{thm}

Note that every set $\Om$ with $\mu(\Om \cap \bdy \Om)=0$ is \p-path almost open,
by Theorem~\ref{thm-path-almost-open-char-intro}.
When $\Om$ is not open and in the metric setting, the Sobolev
and Dirichlet spaces have to be interpreted by means of upper
gradients as in Section~\ref{sect-ug}.

Removable sets for Sobolev spaces is a classical topic, also related
to sets of capacity zero and to singularities of quasiconformal mappings.
We refer to Koskela~\cite{Koskela}
for further references and a much more extensive discussion.
Among other results in \cite{Koskela},
\p-porous sets contained in a hyperplane
  were shown to be removable for $H^{1,p}$
  (and equivalently for the \p-Poincar\'e inequality).

Removable sets for Poincar\'e inequalities in metric spaces were
studied in Koskela--Shanmugalingam--Tuominen~\cite{KoShTu00}.
Their results on porous sets, together with our
Theorems~\ref{thm-A-intro}, \ref{thm-hN1-rem-global}
and~\ref{thm-hN1-rem-local}, therefore provide examples of removable
sets for Sobolev and Dirichlet spaces,
see \cite[Theorems~A,~B and Proposition~3.3]{KoShTu00}.
Removability for Dirichlet spaces was not
discussed in \cite{Koskela} or \cite{KoShTu00}.

As mentioned in \cite[p.~335]{KoShTu00},
Koskela's proof can be generalized to metric spaces with global
assumptions, provided that $E$ is compact, its complement is connected
and $p>1$.
We approach the problem from a different angle, though similar methods 
lie behind some of our arguments as well.
Namely, we
rely on extensions of Newtonian (Sobolev) functions from
a noncomplete metric space $X$ to its completion $\Xhat$, recently
considered in \cite{BBnoncomp} for $p>1$.

To be able to handle also $p=1$, we therefore first prove the following extension result.
In addition,
as in \cite{BBnoncomp},
we replace the global assumptions of a doubling measure and a 
$1$-Poincar\'e inequality by weaker local conditions.
These local assumptions, as well as the
Newtonian and Dirichlet spaces $N^{1,p}(X)$ and $D^p(X)$,  
will be defined in Section~\ref{sect-ug}.

\begin{thm}\label{thm-intro}
	Assume that the doubling property
	and the $1$-Poincar\'e inequality hold within 
	an open set $\Om \subset X$
	in the sense of Definition~\ref{def-local-intro}.
	Let $u \in \Done(\Om)$.
Let $\Omhat=\Xhat \setm \itoverline{X \setm \Om}$,
where the closure is taken in the completion $\Xhat$ of $X$.
Then there is 
	$\uhat\in\Done(\Omhat)$ such that $\uhat=u$ $\ConeX$-q.e.\ in 
	$\Om$ and
	the
	minimal $1$-weak upper gradients $g_{u}:=g_{u,\Om}$ of $u$ and
	$g_{\uhat}:=g_{\uhat,\Omhat}$ of $\uhat$, with
          respect to $\Om$ and $\Omhat$, respectively, satisfy
	\[ 
	g_{\uhat} \le A_{0} g_u \quad \text{a.e.\ in }\Om,
	\] 
	where $A_{0}$ is a constant depending only on the doubling 
	constant and both constants in the $1$-Poincar\'e inequality within 
	$\Om$.
	In particular, the function $\uhat$ can be taken to be
		\begin{equation}\label{eq-Leb-pt-1}
		\uhat(x)=\limsup_{r \to 0} \vint_{\Bhat(x,r) \cap \Om} u \, d\mu, \quad
		x \in \Omhat.
		\end{equation}

If $\Om$ is also $1$-path open in $\Xhat$ 
then we can, in the above conclusion\/
\textup{(}except for~\eqref{eq-Leb-pt-1}\textup{)},        
        take
	$\uhat\equiv u$ 
	and $g_{\uhat}\equiv g_u$
        in $\Om$.
\end{thm}

The idea of the proof is to approximate $u$
by discrete convolutions that immediately extend to $\Omhat$.
This goes back to
	the aforementioned paper by Koskela~\cite[Theorem~C]{Koskela}
and is similar to \cite{BBnoncomp} and
	Heikkinen--Koskela--Tuominen~\cite{HKT}.
When $1<p<\infty$,
one can use the reflexivity of $L^p$
to extract a weakly converging subsequence from the \p-weak
upper gradients of these discrete convolutions. 
In the case $p=1$, we instead  show that the sequence of
$1$-weak upper gradients
is \emph{equi-integrable}, and then apply the Dunford--Pettis theorem to obtain a weakly converging subsequence.
In this way, at the limit we obtain the desired function $\uhat\in \Done(\Omhat)$.
Just as in the case $p>1$ considered in~\cite{BBnoncomp}, we do not know whether
it is ever necessary to have $A_0>1$.

To replace the usual global assumptions by similar local ones in our results, we
apply a recent result of Rajala~\cite{rajala} about approximations by
uniform domains.
In particular, we extend results  about
measurability of functions with $L^p$-integrable upper gradients 
(from \cite{JJRRS}),
about \p-quasiopen, 
  \p-path and \p-finely open sets
(from \cite{BBLat2}, \cite{BBMaly} and~\cite{LahJMPA}),
and about Lebesgue points for $\None$-functions
(from \cite{KKST}),
to spaces that only satisfy local assumptions, see
Section~\ref{sect-rajala} and Proposition~\ref{prop-Leb-pt}.
These localized results are useful later on in the paper.

Observe that in Theorem~\ref{thm-intro} we do not
  require $\Om$ to be measurable in $\Xhat$, see Section~\ref{sect-Xhat}
  for details.
 It is not known if $\Om$ can
  satisfy the assumptions in Theorem~\ref{thm-intro}
  and at the same time 
  be nonmeasurable
in $\Xhat$.
Nevertheless, in
Section~\ref{sec:counterexample} we construct a
measurable
set in $\R^2$,
with full measure and satisfying the
conclusions in Theorems~\ref{thm-A-intro}
and~\ref{thm-intro}
(except for the last part), but which
is not even \p-path almost open in $\R^2$.

The role of \p-path (almost) open sets in our arguments is that
they preserve minimal \p-weak upper gradients and  
sets with zero capacity, see
Lemmas~\ref{lem-gu-on-p-path-open-Xhat}, \ref{lem-same-zero-cap} and 
Bj\"orn--Bj\"orn~\cite[Proposition~3.5]{BBnonopen}.
In Section~\ref{sec:p-path-almost-open}, we study these sets in more
detail and prove the following characterization, which 
combines Theorems~\ref{thm-BBM-gen}
and~\ref{thm-p-path-almost-open-char}.

\begin{thm}        \label{thm-path-almost-open-char-intro}
Assume that $X$ is locally compact and that $\mu$ is locally doubling
and supports a local \p-Poincar\'e inequality, $1 \le p<\infty$.
Let $U \subset X$ be measurable.
Then the following are equivalent\/\textup{:}
\begin{enumerate}
\item \label{h-1}
$U$ is \p-path almost open.
\item  \label{h-2}
$U=V\cup N$, where $V$ is
  \p-path open and $\mu(N)=0$.
\item \label{h-3}
$U=V\cup N$, where $V$ is
  \p-quasiopen and $\mu(N)=0$.
\item \label{h-4}
$U=V\cup N$, where $V$ is
  \p-finely open and $\mu(N)=0$.
\end{enumerate}
\end{thm}

The \p-path almost open sets were introduced in \cite{BBnonopen} and
the implication~\ref{h-2}$\imp$\ref{h-1} was proved therein
\cite[Lemma~3.2]{BBnonopen}. 
Since \p-quasiopen, \p-path and \p-finely open sets are
  measurable (under the above assumptions), the characterization in
  Theorem~\ref{thm-path-almost-open-char-intro} is not possible for
  nonmeasurable \p-path almost open sets.
At the same time, we show that there are nonmeasurable \p-path almost open subsets
of $\R^n$, $n \ge 2$, provided that the continuum hypothesis is
assumed.
Together with Theorem~\ref{thm-path-almost-open-char-intro}
and Example~\ref{ex-Leb+Dirac}, this
answers Open problem~3.4 in~\cite{BBnonopen}.

Quasiopen and finely open sets have earlier been  used in various 
areas of mathematics.
For example, quasiopen sets appear naturally as minimizing sets in shape
optimization problems, see e.g.\ Buttazzo--Dal Maso~\cite{buttazzo-dalMaso},
Buttazzo--Shrivastava~\cite[Examples~4.3 and~4.4]{buttazzo-shr}
and Fusco--Mukherjee--Zhang~\cite{FuscoMZ}.
They are also level sets of Sobolev functions and are thus 
(together with \p-finely open sets)
suitable for
the theory of Sobolev spaces,
see Kilpel\"ainen--Mal\'y~\cite{KiMa92}, Mal\'y--Ziemer~\cite{MZ} and
Fuglede~\cite{Fuglede71}, \cite{Fug}.
In this context, our Theorems~\ref{thm-A-intro},
\ref{thm-hN1-rem-global} and~\ref{thm-hN1-rem-local}
fully characterize removable singularities
with zero measure for Sobolev (and Dirichlet) spaces on \p-quasiopen
(and thus also \p-finely open) sets.
Finely open sets define the fine topology and are closely related to
superharmonic functions. 
Fine potential theory on finely open sets has been studied since the
1940s, see Cartan~\cite{cartan46} (the linear case, $p=2$).

\begin{ack}
The authors wish to thank 
two anonymous referees for a very careful reading of the paper
and for suggesting several improvements.
The first two authors were supported by the Swedish Research Council, 
grants 2016-03424 and 2020-04011 resp.\ 621-2014-3974 and 2018-04106.
This research began
while the third author visited Link\"oping University in 2017 and
2018; he thanks the Department of Mathematics for its warm hospitality.
\end{ack}

\section{Upper gradients and Newtonian spaces}
\label{sect-ug}

\emph{We assume throughout the paper, except for Section~\ref{sect-rem}, 
that $1\le p<\infty$ and that
$X=(X,d,\mu)$ is a metric space equipped
with a metric $d$ and a positive complete  Borel  measure $\mu$
such that $0<\mu(B)<\infty$ for all 
balls $B \subset X$.}

\medskip

It follows that $X$ is  separable and Lindel\"of.
To avoid pathological situations we assume that $X$ contains
at least two points.
Proofs of the results in 
this section can be found in the monographs
Bj\"orn--Bj\"orn~\cite{BBbook} and
Heinonen--Koskela--Shanmugalingam--Tyson~\cite{HKSTbook}.

A \emph{curve} is a continuous mapping from an interval,
and a \emph{rectifiable} curve is a curve with finite length.
Unless said otherwise, we will only consider curves that
are nonconstant, compact and 
rectifiable, and thus each curve can 
be parameterized by its arc length $ds$. 
A property is said to hold for \emph{\p-almost every curve}
if it fails only for a curve family $\Ga$ with zero \p-modulus. 
Here the \emph{\p-modulus} of $\Gamma$ is 
\[
\Mod_{p,X}(\Gamma):=\inf_\rho \int_X\rho^p\, d\mu
\]
with the infimum taken over all nonnegative 
Borel functions $\rho$ on $X$ such that 
$\int_\ga\rho\, ds\ge 1$ for each $\ga\in\Ga$.

Following Heinonen--Kos\-ke\-la~\cite{HeKo98},
we next introduce upper gradients 
(called very weak gradients in~\cite{HeKo98}).

\begin{deff} \label{deff-ug}
A Borel function
$g\colon X \to [0,\infty]$
is an \emph{upper gradient} 
of a function
$u\colon X \to \eR:=[-\infty,\infty]$
if for all  curves  
$\gamma \colon [0,l_{\gamma}] \to X$,

\begin{equation} \label{ug-cond}
        |u(\gamma(0)) - u(\gamma(l_{\gamma}))| \le \int_{\gamma} g\,ds,
\end{equation}
where the left-hand side is considered to be $\infty$ 
whenever at least one of the 
terms therein is infinite.
If
$g\colon X \to [0,\infty]$
is measurable 
and \eqref{ug-cond} holds for \p-almost every curve,
then $g$ is a \emph{\p-weak upper gradient} of~$u$. 
\end{deff}

The \p-weak upper gradients were introduced in
Koskela--MacManus~\cite{KoMc}. 
It was also shown therein
that if $g \in \Lploc(X)$ is a \p-weak upper gradient of $u$,
then one can find a sequence $\{g_j\}_{j=1}^\infty$
of upper gradients of $u$ such that $\|g_j-g\|_{L^p(X)} \to 0$.
If $u$ has an upper gradient in $\Lploc(X)$, then
it has an a.e.\ unique \emph{minimal \p-weak upper gradient} $g_u \in \Lploc(X)$
in the sense that $g_u \le g$ a.e.\ for every
\p-weak upper gradient $g \in \Lploc(X)$ of $u$,
see Shan\-mu\-ga\-lin\-gam~\cite{Sh-harm}
and Haj\l asz~\cite{Haj03}.
Following Shanmugalingam~\cite{Sh-rev}, 
we define a version of Sobolev spaces on the metric space $X$.

\begin{deff} \label{deff-Np}
For a measurable function
$u\colon X\to \eR$,
let 
\[
        \|u\|_{\Np(X)} = \biggl( \int_X |u|^p \, d\mu 
                + \inf_g  \int_X g^p \, d\mu \biggr)^{1/p},
\]
where the infimum is taken over all upper gradients $g$ of $u$.
The \emph{Newtonian space} on $X$ is 
\[
        \Np (X) = \{u :  \|u\|_{\Np(X)} <\infty \}.
\]
\end{deff}

The quotient
space $\Np(X)/{\sim}$, where  $u \sim v$ if and only if $\|u-v\|_{\Np(X)}=0$,
is a Banach space and a lattice, see Shan\-mu\-ga\-lin\-gam~\cite{Sh-rev}.
We also define
\[
   \Dp(X)=\{u : u \text{ is measurable, finite
a.e.\ and  has an upper gradient
     in }   L^p(X)\}.
\]
This definition deviates from the definition
in \cite[Definition~1.54]{BBbook} in that it requires the functions
to be finite a.e., which   
will be useful in e.g.~Theorem~\ref{thm-isom-rem}, 
see Remark~\ref{rmk-Dp}.
The two definitions coincide whenever $X$ supports a local \p-Poincar\'e
inequality, since any measurable function with an upper gradient in $L^p(X)$ then
belongs to $L^1\loc(X)$,
 see \cite[Proposition~4.13]{BBbook} and \cite[p.~50]{BBnoncomp}.

In this paper we assume that functions in $\Np(X)$
and $\Dp(X)$
 are defined everywhere (with values in $\eR$),
not just up to an equivalence class in the corresponding function space.
This is important for upper gradients to make sense.

For a measurable set $E\subset X$, the Newtonian space $\Np(E)$ is defined by
considering $(E,d|_E,\mu|_E)$ as a metric space in its own right.
We say  that $u \in \Nploc(E)$ if
for every $x \in E$ there exists a ball $B_x\ni x$ such that
$u \in \Np(B_x \cap E)$.
The spaces $L^p(E)$, $\Lploc(E)$,
$\Dp(E)$ and $\Dploc(E)$ are defined similarly.
If $u,v \in \Dploc(X)$, then $g_u=g_v$ a.e.\ in the set $\{x \in X :  u(x)=v(x)\}$,
in particular for $c \in \R$ we have
$g_{\min\{u,c\}}=g_u \chi_{\{u < c\}}$ a.e.
Moreover, if $u,v \in \Dp(X)$, then
  $|u|g_v+|v|g_u$ is a \p-weak upper gradient of $uv$.

It is easily verified by gluing curves together that if $g_1$
  and $g_2$ are upper gradients for a function $u$ in the open sets
  $G_1$ and $G_2$, respectively, then $g_1\chi_{G_1}+g_2\chi_{G_2}$ is
  an upper gradient for $u$ in $G_1\cup G_2$.
  From this it
  follows that if $u\in \Np(G_j)$, $j=1,2$, then $u\in\Np(G_1\cup G_2)$.
A similar sheaf property holds for $D^p$.

\begin{deff} \label{deff-Sob-cap}
The (Sobolev) \emph{capacity}  of a set $E\subset X$  is the number 
\begin{equation*} 
   \CpX(E)=\inf_u    \|u\|_{\Np(X)}^p,
\end{equation*}
where the infimum is taken over all $u\in \Np (X)$ such that $u=1$ on $E$.
\end{deff}

We say that a property holds \emph{$\CpX$-quasieverywhere} ($\CpX$-q.e.)\ 
if the set of points  for which the property fails 
has zero $\CpX$-capacity.
The capacity is the correct gauge 
for distinguishing between two Newtonian functions. 
Namely, if $v \in \Np(X)$ then $u \sim v$ if and only if $u=v$ $\CpX$-q.e.
Moreover, if $u,v \in \Dploc(X)$ and $u=v$ a.e., then $u=v$ $\CpX$-q.e.
In this paper we will use many different $\CpX$-capacities with respect
to different metric spaces $X$; this will always be carefully denoted
in the superscript.

\begin{deff}   \label{deff-q-cont}
An $\eR$-valued function $u$,
defined on a set  $E \subset X$, is \emph{$\CpX$-quasi\-con\-tin\-u\-ous}
if for every $\eps>0$ there is an open set $G\subset X$
such that $\CpX(G)<\varepsilon$ and $u|_{E \setm G}$ is
$\R$-valued
and continuous.
\end{deff}

For a ball $B=B(x,r)$ with centre $x$ and radius $r$, we let
$\lambda B = B(x, \lambda r)$. In metric spaces
it can happen that balls with different centres and/or radii denote the same set.
We will, however,
make the convention that a ball
comes with a predetermined centre and radius.
In this paper,
balls are assumed to be open.

\begin{deff} \label{def-local-intro}
The measure $\mu$ is \emph{doubling within  an open set $\Om\subset X$}
if there is $C>0$ (depending on $\Om$) 
such that $\mu(2B)\le C \mu(B)$ for all balls $B \subset \Om$.

Similarly, the
\emph{\p-Poincar\'e inequality holds within an open set $\Om$}
if there are constants $C>0$ and $\lambda \ge 1$
(depending on $\Om$)
such that for all balls $B\subset \Om$,
all integrable functions $u$ on $\la B$, and all upper gradients $g$ of $u$
in $\la B$,
\begin{equation}  \label{eq-PI-on-B}
        \vint_{B} |u-u_B| \,d\mu
        \le C r_B \biggl( \vint_{\lambda B} g^{p} \,d\mu \biggr)^{1/p},
\end{equation}
where $u_B:=\vint_B u \,d\mu := \int_B u\, d\mu/\mu(B)$
and $r_B$ is the radius of $B$.

Each of these properties is
called \emph{local} if 
for every $x \in X$ there is $r>0$
(depending on $x$) such that the
property holds within $B(x,r)$.
The property is
called \emph{semilocal} if it holds
within every ball
$B(x_0,r_0)$ in $X$.
If moreover $C$ and $\la$ are independent of $x_0$ and $r_0$, then
it is called \emph{global}.
\end{deff}

Note that there is a difference between a property holding
\emph{within} $\Om\subset X$ (i.e.\ for balls taken in the 
underlying space $X$) and \emph{on} 
$\Om$, seen as a metric space in its own right, where balls are 
taken with respect to 
$\Om$.

The \p-Poincar\'e inequality can equivalently be defined using
  \p-weak upper gradients.
We will need the following characterization of
the \p-Poincar\'e inequality showing that it is enough to test
with bounded $u \in \Np(X)$.

\begin{lem} \label{lem-PI-char}
Let $\Om \subset X$ be open.
Assume that
there are constants $C>0$ and $\lambda \ge 1$ such that~\eqref{eq-PI-on-B}
holds for all balls $B\subset \Om$ and all bounded $u \in \Np(X)$.
Then the \p-Poincar\'e inequality holds within $\Om$
with the same constants $C$ and $\la$.
\end{lem}

Below and later, we write $u_\limplus=\max\{0,u\}$ and $u_\limminus=\max\{0,-u\}$.

\begin{proof}
Let $B=B(x,r)\subset \Om$ be a ball,
$u$ be an integrable function on $\la B$, and $g$ be an upper gradient
of $u$ in $\la B$.
We may assume that $g \in L^p(\la B)$, as otherwise
there is nothing to prove.
Thus $u \in D^p(\la B)$.

For $j=1,2,\ldots$\,, let $B_j=B(x,(1-2^{-j})r)$,
\[
  u_j=\max\{\min\{u,j\},-j\}
  \quad \text{and} \quad
  v_j=(1-2^{j+1}r\dist(x,\la B_j))_\limplus u_j,
\]
extended by zero outside $\la B$. 
Then $v_j \in \Np(X)$ is bounded
and $g$ is an upper gradient of $v_j$ in $\la B_j$.
  Thus~\eqref{eq-PI-on-B} applied
to $v_j$ and $B_j$ gives
\[ 
         \vint_{B_j} |u_j-(u_j)_{B_j}| \,d\mu
         = \vint_{B_j} |v_j-(v_j)_{B_j}| \,d\mu 
         \le C r_{B_j} \biggl( \vint_{\lambda B_j} g^{p} \,d\mu \biggr)^{1/p}.
\]
The result now follows
from the fact that $u_j\to u$ in $L^1(B)$ as $j \to \infty$.
\end{proof}

\section{From global to local assumptions}
\label{sect-rajala}

In this section we show how a recent result due to 
Rajala~\cite[Theorem~1.1]{rajala} can be used to lift results, that
have earlier been obtained under global assumptions, to spaces with only
local assumptions.
This will be useful later in our considerations.
The main idea in this localization approach is to see 
suitable neighbourhoods of points in $X$ as ``good'' metric spaces 
in their own right.
Since balls may be disconnected and need not support a 
Poincar\'e inequality, they do not in general serve as such
good neighbourhoods. 
Even when a ball, or its closure,  is connected it can fail to support a 
Poincar\'e inequality and the measure may fail to be globally doubling on it.
Instead, closures of
the uniform domains 
constructed by Rajala~\cite{rajala} 
will do the job, since they are compact, support global Poincar\'e inequalities and
the measure is globally doubling on them.

Recall that a space is \emph{geodesic}
if every pair of points can be connected by a curve whose
length equals the distance between the points, and that
a \emph{domain} is an open connected set.
A domain $G\subset X$ 
is  \emph{uniform} if there is a constant $A\ge1$
such that  for every pair $x,y\in G$ there is a 
curve $\ga: [0,l_\ga] \to G$ with $\ga(0)=x$ and
$\ga(l_\ga)=y$ such that $l_\ga \le A d(x,y)$ and
\[
\dist(\ga(t),X \setm G) \ge \frac{1}{A} \min\{t, l_\ga-t\}
    \quad \text{for } 0 \le t \le l_\ga.
\]
As usual, $\dist(x,\emptyset)=\infty$.
Moreover, $X$ is \emph{globally doubling}
if there is a constant $N$ such that every ball $B(x,r)$ 
can be covered by 
$N$ balls with radii $\frac{1}{2}r$. 

The following result was proved in~\cite{rajala}
under the assumption  that $X$ is \emph{quasiconvex}.
In particular, it applies to geodesic spaces because their
quasiconvexity constant is $1$.

\begin{thm} \label{thm-rajala}
\textup{(Rajala~\cite[Theorem~1.1]{rajala})}
Let $X$ be a geodesic metric space
and let $U \subset X$ be a bounded domain.
If $U$ is globally doubling and $\eps>0$,
then there is a uniform domain $G$ such that
\[
    \{x \in U : \dist(x,X \setm U) \ge \eps\} \subset G \subset U.
\]
\end{thm}

Note that if $U=X$ then $X$ itself is a uniform domain and $G=U=X$ above.
In the definition of uniform domains it is often assumed
that
$G \subsetneq X$.
Allowing $G=X$, as in \cite{rajala},
is convenient
when formulating Theorems~\ref{thm-rajala} and~\ref{thm-rajala-iter}.

In
\cite{rajala} it is assumed that $X$ is globally doubling, and approximation
from outside by uniform domains is also deduced. 
However, when approximating from inside as in Theorem~\ref{thm-rajala},
it is easy to see that in the proof in \cite{rajala}
it is enough to apply \cite[Lemma~2.1]{rajala} with respect to $U$.
It is therefore enough to assume that $U$ is globally doubling, which 
makes it possible to deduce the following result.

\begin{thm} \label{thm-rajala-iter}
Assume that 
the \p-Poincar\'e inequality and the 
doubling property for $\mu$ hold
{\rm(}with constants $\CPI$, $\la$ and $C_\mu${\rm)} within $B_1=B(x_1,r_1)$.
Also assume that $\itoverline{\La B}_1$ is compact,
where $\La=3C_\mu^3\CPI$.

Then there is a bounded uniform domain $G$ such that
\[
\tau B_1 \Subset G \subset \tfrac16 B_1,
\quad \text{where }
\tau=\frac{1}{60\Lambda}.
\]

Moreover, $\mu|_G$ and $\mu|_{\clG}$ are
globally doubling and support
global \p-Poincar\'e inequalities
on the metric spaces $G$ and $\clG$, respectively.
\end{thm}

Note that $\Lambda$ is independent of $\la$,
as in \cite[Lemma~4.9]{BBsemilocal} and
\cite[Theorem~4.32]{BBbook}.
As usual, by $E \Subset G$ we mean that $\itoverline{E}$
is a compact subset of $G$.

\begin{remark}  \label{rem-mu(bdy)=0}
Note that $G$, being uniform, satisfies the 
so-called \emph{corkscrew condition}.
Applying Theorem~2.8 in
Bj\"orn--Shan\-mu\-ga\-lin\-gam~\cite{BjShJMAA} to $A=\clG$ and
letting $\rho\to0$ in \cite[(2.2)]{BjShJMAA} shows that $\mu(\bdy G)=0$.
\end{remark}

\begin{proof}[Proof of Theorem~\ref{thm-rajala-iter}]
Define the \emph{inner metric}
\[
    \dinK(x,y)=\inf \length(\ga),
\]
where the infimum is taken over all 
curves $\ga \subset\itoverline{\La B}_1$ connecting $x$ and $y$.
Let 
\[
Y=\{x \in \itoverline{\La B}_1 : \dinK(x,x_1)<\infty\}
\]
be the rectifiably connected component of $\itoverline{\La B}_1$
  containing $x_1$.
As $\itoverline{\La B}_1$ is compact,
it follows from Ascoli's theorem that $(Y,\dinK)$ is a geodesic
metric space.
By Bj\"orn--Bj\"orn~\cite[Lemma~4.9]{BBsemilocal},
  every pair of points $x,y \in \tfrac15 B_1$ can be connected by a curve
in $\itoverline{\La B}_1$,
of length at most $9\Lambda  d(x,y)$.
Hence, both $\tfrac16 B_1$ and 
\[
\Bin:=  \{x \in Y:\dinK(x,x_1)<\tfrac16 r_1\}
\]
are open and 
$\tau B_1 \Subset \Bin \subset \tfrac16 B_1 \subset Y$.
The reason for using the inner metric is that inner balls are always connected, while
standard balls, such as $\tfrac16 B_1$, need not be connected.

By \cite[Proposition~3.4]{BBsemilocal}, the ball
$\tfrac16 B_1$   is globally doubling.
As $d$ and $\dinK$ are comparable within $\tfrac16 B_1$,
also $(\Bin,\dinK)$ is globally doubling.
Since $(\Bin,\dinK)$ is connected, we
can therefore apply Theorem~\ref{thm-rajala} with respect to $(Y,\dinK)$
and obtain a uniform domain $G$ such that $\tau B_1 \Subset G \subset \Bin$.
Note that since $d$ and $\dinK$ are comparable within $\tfrac15 B_1$
and $\dist(\tfrac16 B_1, X \setm Y) \ge \tfrac{1}{30}r_1$, uniformity
is the same with respect to 
$(Y,\dinK)$ and $(X,d)$
(although the uniformity
constants may be different).

By Bj\"orn--Shan\-mu\-ga\-lin\-gam~\cite[Lemmas~2.5 and~4.2]{BjShJMAA},
$\mu|_G$ is globally doubling on $G$. 
Next, we use
\cite[Theorem~4.4]{BjShJMAA}, to
see that $\mu|_G$ supports a global \p-Poincar\'e inequality on $G$.
Since the proof of \cite[Theorem~4.4]{BjShJMAA} only uses balls
contained (together with their dilations)
in $G$, the proof applies under our assumptions.
As
$\mu(\bdy G)=0$, by Remark~\ref{rem-mu(bdy)=0},
the same
conclusions hold for $\mu|_{\clG}$.
(To see that the Poincar\'e inequality holds on $\clG$,
  one only needs to observe that if
  $g$ is an upper gradient of $u$ on $\clG$, then $g|_G$ is an
  upper gradient of $u|_G$ on $G$,
    see also \cite[Proposition~3.6]{BBnoncomp} for further details.)
\end{proof}

One result 
that can be obtained using Theorem~\ref{thm-rajala-iter} 
is the following extension of Theorem~1.11 in
J\"arvenp\"a\"a--J\"arvenp\"a\"a--Rogovin--Rogovin--Shan\-mu\-ga\-lin\-gam~\cite{JJRRS}.
Since local assumptions are inherited by open subsets, it
directly applies
also to open $\Om \subset X$ (in place of $X$), cf.\ Remark~\ref{rem-loc-spcs}.

\begin{thm} \label{thm-JJRRS-gen}
Assume that $X$ is locally compact and that $\mu$ is locally doubling
and supports a local \p-Poincar\'e inequality.
Let $g \in L^p\loc(X)$ be an upper gradient
of $u\colon X \to \eR$.
Then $u \in L^p\loc(X)$ and $u$ is in particular measurable.
\end{thm}

\begin{proof}
Let $x \in X$.
It follows from Theorem~\ref{thm-rajala-iter} that there is a 
bounded uniform domain $G_x \ni x$ such that
$\clGx$ is compact and
$\mu|_{\clGx}$ is globally doubling and supports a global \p-Poincar\'e inequality
on $\clGx$.
In particular, $g|_{\clGx}\in L^p(\clGx)$ is an upper gradient of 
$u|_{\clGx}$ in $\clGx$ and Theorem~1.11 in~\cite{JJRRS}
shows that $u|_{\clGx}$ is measurable and belongs to $L^p\loc(\clGx)$.

As $X$ is Lindel\"of it follows that $u$ is measurable on $X$ and
$u\in L^p\loc(X)$.
\end{proof}

Another consequence of Rajala's theorem is a characterization of
  \p-path open sets under local assumptions.
These sets will play a prominent role in our
studies, since they preserve minimal \p-weak upper gradients and 
sets with zero capacity, see 
Lemmas~\ref{lem-gu-on-p-path-open-Xhat} and~\ref{lem-same-zero-cap} below and 
Bj\"orn--Bj\"orn~\cite[Proposition~3.5]{BBnonopen}.
The relation between \p-path open and \p-path almost open sets will be
studied in Section~\ref{sec:p-path-almost-open}.

\begin{deff}   \label{def-p-path-open}
A set $G\subset A $ 
is \p-\emph{path open} in $A \subset X$
if for \p-almost every curve
$\ga\colon [0,l_\ga]\to A$,
the set $\ga^{-1}(G)$ 
is\/ \textup{(}relatively\/\textup{)} open in $[0,l_\ga]$.

Further, $G\subset A $ 
is 
\p-\emph{path almost open}
in $A \subset X$ if for \p-almost every curve
$\ga\colon [0,l_\ga]\to A$,
the set $\ga^{-1}(G)$ is the union
of an open set and a set with zero $1$-dimensional Lebesgue measure.
\end{deff}

The \p-modulus $\Modp(\Ga)$ of the exceptional curve family $\Ga$ can equivalently
be measured within $X$ or $A$, provided that $A$ is equipped with the appropriate
restriction $\mut$ of $\mu$ to $A$. Since $A$ may be nonmeasurable, $\mut$ is
defined by letting
\begin{equation*}
    \mut(\Et)=\inf \{\mu(E): E \supset \Et \text{ and $E$ 
is a Borel set with respect to $X$}\}
\end{equation*}
for Borel sets $\Et$ in $A$, and then completing $\mut$.
This makes $\mut$ into a complete Borel regular measure on $A$, which
coincides with the restriction
  $\mu|_A$ when $A$ is $\mu$-measurable.
It also follows that every Borel function $\rhot$ on $A$ has a Borel extension
$\rho$
to $X$ such that
\[
            \int_A \rhot\,d\mut = \int_X \rho\,d\mu.
\]
Hence $\Mod_{p,A}(\Ga)=\Mod_{p,X}(\Ga)$ as claimed.
The relation between $\mut$ and $\mu$ is quite similar to the relation
between $\mu$ and $\muhat$ as discussed in the beginning of
Section~\ref{sect-Xhat} and 
in the corrigendum of  Bj\"orn--Bj\"orn~\cite{BBnoncomp},
and the relation between $\mu_X$ and $\mu_Y$ in Section~\ref{sect-rem}.

The two properties in Definition~\ref{def-p-path-open}
  are transitive, as shown by the following result.
Note also that it follows from \cite[Proposition~2.45]{BBbook} that if
 $1 \le p < q$ and $G$ is $q$-path open (resp.\ $q$-path almost open) in $X$,
  then $G$ is \p-path open (resp.\ \p-path almost open) in $X$.

\begin{lem} \label{lem-trans}
Assume that $G_1 \subset G_2 \subset G_3$ and that $G_2$ is \p-path
almost open in $G_3$.
Then $G_1$ is \p-path almost open in $G_2$ if and only if it is
\p-path almost open in $G_3$.

The corresponding result also holds if ``\p-path almost open'' is replaced
by ``\p-path open'' throughout.
\end{lem}

\begin{proof}
If $G_1$ is \p-path almost open in $G_3$, then
it is \p-path almost open in $G_2$
(in view of the discussion above) simply because
every curve in $G_2$ is a curve in $G_3$.

Conversely, assume that $G_1$ is \p-path almost open in $G_2$.
Let $\Ga_j$, $j=1,2$, be the family of curves $\ga$ in $G_{j+1}$ such
that $\ga^{-1}(G_j)$ is not a union of an open set and
a set of measure zero.
Let $\Ga'$ be the family  of curves in $G_3$ which 
contain a subcurve in $\Ga_1$.
Then by assumption, \cite[Lemma~1.34\,(c)]{BBbook} and the discussion above,
\[
\Mod_{p,G_3}(\Ga') \le\Mod_{p,G_3}(\Ga_1)=\Mod_{p,G_3}(\Ga_2)=0. 
\]

Next take a curve $\ga \colon [0,l_\ga]\to G_3$ such that
$\ga \notin \Ga_2 \cup \Ga'$.
Then $\ga^{-1}(G_2)$ is a union of an open set $A$ and a set of
measure zero.
Since $A \subset \R$, it  can be written as a countable
or finite union of pairwise disjoint open intervals $A_j$.
Each $A_j$ can be written as an increasing countable union of compact intervals,
and since $\ga \notin \Ga'$ we see that $A_j \cap \ga^{-1}(G_1)$ is
a union of an open set and a set of measure zero.
Hence $\ga^{-1}(G_1)$ is a union of an open set  and a set of
measure zero.
As $\Mod_{p,G_3}(\Ga_2 \cup \Ga')=0$, we have shown that $G_1$ is
\p-path almost open in $G_3$.

The \p-path open case is similar.
\end{proof}

Next, we shall
characterize \p-path open sets in terms
of \p-quasiopen and \p-finely open sets, under local assumptions.
Such characterizations have been done under global assumptions,
and as earlier in this section we will show how to ``lift'' them
to local assumptions.

A set $V\subset X$ is \emph{\p-quasiopen} if for every
$\varepsilon>0$ there is an open set $G\subset X$ such that
$\CpX(G)<\varepsilon$ and $G\cup V$ is open. 
Every \p-quasiopen set is measurable by \cite[Lemma~9.3]{BBnonopen}.
The family of \p-quasiopen sets does not form a topology (in general) but
it is closed under countable unions.

If 
$E \subset A$ are bounded subsets of $X$, then 
the \emph{variational capacity} of $E$ with respect to $A$ is
\begin{equation*} 
\cpX(E,A) = \inf_u\int_{X} g_u^p\, d\mu,
\end{equation*}
where the infimum is taken over all $u \in \Np(X)$
such that $u\geq 1$ on $E$ and $u=0$ 
on  $X  \setm A$.
(If no such function $u$ exists then $\cpX(E,A)=\infty$.)

A set $E\subset X$ is  \emph{\p-thin} at $x\in X$ if
\begin{equation}   \label{deff-thin}
\int_0^1\biggl(\frac{\cpX(E\cap B(x,r),B(x,2r))}{\cpX(B(x,r),B(x,2r))}\biggr)^{1/(p-1)}
     \frac{dr}{r}<\infty
\end{equation}
when $p>1$, and if
\begin{equation}   \label{deff-thin-p1}
	\lim_{r\to 0}\frac{\capp_1^X(E\cap B(x,r),B(x,2r))}{\capp_1^X(B(x,r),B(x,2r))}=0
	\end{equation}
when $p=1$.
(The
quotients in~\eqref{deff-thin} and~\eqref{deff-thin-p1} are
interpreted as $1$ if the denominators
therein are zero.)
Note that,
under the assumptions of Theorem~\ref{thm-BBM-gen} below,
 $\cpX(B(x,r),B(x,2r))$ is comparable to
$\mu(B(x,r))/r^p$ 
for sufficiently small $r$,
by
e.g.\ \cite[the proof of Proposition~6.16]{BBbook}, 
and so the latter quantity could also be
used
in~\eqref{deff-thin} and~\eqref{deff-thin-p1},
as was done in e.g.\
Lahti~\cite{LahJMPA}.

A set $V\subset X$ is \emph{\p-finely open} if
$X\setminus V$ is \p-thin at each point $x\in V$.
The family of \p-finely open sets forms the \p-fine topology.

The following theorem gives the
equivalence of \ref{h-2}--\ref{h-4} in Theorem~\ref{thm-path-almost-open-char-intro},
since $\mu(Z)=0$ whenever $\CpX(Z)=0$ 
(this follows directly from Definition~\ref{deff-Sob-cap}).

\begin{thm} \label{thm-BBM-gen}
Assume that $X$ is locally compact and that $\mu$ is locally doubling
and supports a local \p-Poincar\'e inequality.
Let $U\subset X$.
Then the following are equivalent.
\begin{enumerate}
\item \label{it-path}  $U$ is \p-path open.
\item \label{it-quasi} $U$ is \p-quasiopen.
\item \label{it-pfine} $U=V\cup Z$, where $V$ is \p-finely open and $\CpX(Z)=0$.
\end{enumerate}
\end{thm}

When $X$ is complete and $\mu$ is globally doubling
and supports a global \p-Poincar\'e inequality,
these characterizations are due to
  Bj\"orn--Bj\"orn--Latvala~\cite[Theorem~4.9]{BBLat1}, \cite[Theorem~1.4]{BBLat2},
  Bj\"orn--Bj\"orn--Mal\'y~\cite[Theorem~1.1]{BBMaly},
  Shanmugalingam~\cite[Remark~3.5]{Sh-harm},
  and Lahti~\cite[Corollary~6.12]{LahJMPA}
  combined with Hakkarainen--Kinnunen~\cite[Theorems~4.3 and~5.1]{HaKi}.
We will use these results, and the proof below 
just shows how to lift them to local assumptions, without repeating the arguments.

\begin{proof}
We start by some preliminary observations.
By Theorem~\ref{thm-rajala-iter}, for every $x \in X$ there is a 
bounded uniform domain $G_x\ni x$ such that $\clGx$ is compact and
$\mu|_{\clGx}$ is globally doubling and supports a global \p-Poincar\'e inequality
on $\clGx$.
As $X$ is Lindel\"of, there is a countable  cover 
$\{G_{j}\}_{j=1}^\infty$ of $X$, where $G_j=G_{x_j}$.

We also note for later use that Proposition~3.3 in~\cite{BBMaly}
(applied both to $\clGj$ and to $X$ as the underlying space) implies that
$U\cap G_j$ is \p-quasiopen with respect to $\clGj$ 
if and only if it is \p-quasiopen with respect to $G_j$, 
which in turn is equivalent to it being \p-quasiopen with respect to $X$.

\ref{it-path} $\imp$ \ref{it-quasi}
For each $j$, the set $U_j:=U \cap G_j$ is \p-path open in  $\clGj$.
By Theorem~1.1
in~\cite{BBMaly},
we see that $U_j$ is \p-quasiopen
with respect to $\clGj$, and by the above argument also with respect to $X$.
Hence $U = \bigcup_{j=1}^\infty U_j$ is \p-quasiopen in~$X$.

\ref{it-quasi} $\imp$ \ref{it-path} 
This is proved in
Shanmugalingam~\cite[Remark~3.5]{Sh-harm},
without any assumptions on $X$.

To prove the equivalence with \ref{it-pfine}, note that
in the case $p>1$,
a set $W\subset G_j$ is \p-finely open with respect to $\clGj$ 
if and only if for every $x\in W$,
\[
\int_0^{r_x}\biggl(\frac{\cpclGj(B(x,r) \setm W,B(x,2r))}
           {\cpclGj(B(x,r),B(x,2r))}\biggr)^{1/(p-1)}
  \frac{dr}{r} <\infty,
\]
where $r_x>0$ is such that $B(x,2r_x)\subset G_j$.
Clearly, for $0<r<r_x$ and  $A \subset B(x,r)$,
\begin{equation} \label{eq-cpclGj}
\cpclGj(A,B(x,2r))=\cpX(A,B(x,2r)),
\end{equation}
and hence $W$ is \p-finely open with respect to $\clGj$ if and only if
it is \p-finely open with respect to $X$.
The equality~\eqref{eq-cpclGj}
  holds also in the case $p=1$ and implies directly that
$W\subset G_j$
is $1$-finely open with respect to $\clGj$ if and only if
	it is $1$-finely open with respect to $X$.

\ref{it-quasi} $\imp$ \ref{it-pfine}
By the above argument, $U_j:= U\cap G_j$ is \p-quasiopen with respect to $\clGj$.
Theorem~4.9 in~\cite{BBLat1}
(for $p>1$) and~\cite[Corollary~6.12]{LahJMPA} 
combined with~\cite[Theorems~4.3 and~5.1]{HaKi} (for $p=1$)
show that
it can be written as $U_j = V_j\cup Z_j$, where $V_j$ is
\p-finely open with respect to $\clGj$ (and equivalently $X$) and $\CpclGj(Z_j)=0$.
Hence $\bigcup_{j=1}^\infty V_j$ is \p-finely open with respect to $X$.
Moreover, it follows from e.g.\ \cite[Lemma~2.24]{BBbook} that 
the capacities  $\CpclGj$ and $\CpX$ have the same zero
sets in $G_j$ and so $\CpX(\bigcup_{j=1}^\infty Z_j)=0$.
Since $U=\bigcup_{j=1}^\infty V_j \cup \bigcup_{j=1}^\infty Z_j$, \ref{it-pfine} holds.

\ref{it-pfine} $\imp$ \ref{it-quasi}
For each $j$, the set $V\cap G_j$ is \p-finely open in $X$ and thus in $\clGj$, by the
above observation.
Also 
\[
\CpclGj(Z\cap G_j)\le\CpX(Z)=0.
\]
It then follows 
from~\cite[Theorem~1.4]{BBLat2}
(for $p>1$)
 and~\cite[Corollary~6.12]{LahJMPA}
  combined 
with~\cite[Theorems~4.3 and~5.1]{HaKi} (for $p=1$)
that the set
\(
U_j:=(V\cap G_j) \cup (Z\cap G_j)
\)
is \p-quasiopen with
respect to $\clGj$, and thus also with respect to $X$, by the above argument.
Hence, $U=\bigcup_{j=1}^\infty U_j$ is \p-quasiopen in $X$.
\end{proof}

Theorems~1.2 and~1.3 in~\cite{BBMaly} can be extended similarly.
See also Corollary~\ref{cor-qcont} below and
Bj\"orn--Bj\"orn~\cite[Theorem~9.1]{BBsemilocal}.

\section{Extending \texorpdfstring{$N^{1,1}$}{N}-functions
to the completion \texorpdfstring{$\Xhat$}{X}}
\label{sect-Xhat}

The main goal of this section is to prove
  Theorem~\ref{thm-intro}.
We let $\Xhat$ be the completion of $X$ with respect to the metric $d$.
The metric immediately extends to $\Xhat$.
We extend the measure to $\Xhat$ by defining
\[
\muhat(E)=\mu(E\cap X)\quad \text{for every Borel set }E\subset\Xhat,
\]
and then
complete it to obtain a Borel regular measure $\muhat$.
Saksman~\cite[Lemma~1]{saksman} used a similar
construction when studying globally doubling measures.

Now $\Xhat\setminus X$ either has zero $\muhat$-measure or is
$\muhat$-nonmeasurable. 
In both cases, $\muhatin(\Xhat \setm X)=0$,
where the inner measure $\muhatin$ is defined by
\begin{align}
  \muhatin(E)&=\sup \{ \muhat(A): A \subset E \text{ is $\muhat$-measurable}\}
    \nonumber\\
  & =\sup \{ \muhat(A): A \subset E \text{ is a Borel set in $\Xhat$}\}.
  \label{eq-muhat-Borel}
\end{align}
The latter equality follows from
the fact that $\muhat$ is a complete Borel regular measure.
Moreover,
\[
\muhat(E)=\mu(E\cap X)\quad \text{for every }\muhat\text{-measurable set }
E\subset\Xhat,
\]
and thus for $E\subset X$ we have
\begin{equation} \label{eq-mu-iff-muhat=0}
\mu(E)=0
\quad \text{if and only if}
\quad \muhat(E)=0.
\end{equation}

It also follows that every 
$\mu$-measurable (resp.\ Borel) function $u\colon X \to \eR$
has a $\muhat$-measurable (resp.\ Borel)  extension $\uhat\colon \Xhat \to \eR$
such that $\uhat|_X=u$ and 
\begin{equation}  \label{eq-int-X-Xhat-equal}
\int_{\Xhat} \uhat \, d\muhat = \int_X u\,d\mu,
\end{equation}
whenever at least one of the integrals exists.
Conversely, it follows from the above definition of $\muhat$ that
$v|_X$ is $\mu$-measurable (resp.\ Borel) and
\begin{equation}  \label{eq-int-X-Xhat-equal2}
\int_X v\,d\mu = \int_{\Xhat} v \, d\muhat,
\end{equation}
whenever $v\colon\Xhat \to \eR$ is $\muhat$-measurable (resp.\ Borel)
and one of the integrals exists.
See the corrigendum of
 Bj\"orn--Bj\"orn~\cite{BBnoncomp} for further details;
the $\muhat$-nonmeasurable case was unfortunately overlooked in the original
paper.

The following  two auxiliary results
relate notions on $\Xhat$
  to the same notions on \p-path (almost) open sets.

\begin{lem}    \label{lem-gu-on-p-path-open-Xhat}
  Assume that $\Om \subset \Xhat$ is $\muhat$-measurable
    and \p-path almost open in $\Xhat$, $p\ge1$, 
and that $u \in \Dploc(\Om)$.
If $G\subset \Om\cap X$
is $\mu$-measurable and \p-path almost open in $\Om$, 
then the minimal \p-weak upper gradients $g_{u,G}$ and $\ghat_{u,\Om}$
of $u$ with respect to $(G,\mu)$
and $(\Om,\muhat)$, respectively, coincide a.e.\ in $G$.
\end{lem}

Note that by Lemma~\ref{lem-trans}, $G$ is \p-path almost
  open in $\Om$ if and only if it is \p-path almost
  open in $\Xhat$.

\begin{proof}
This is proved verbatim as in
Proposition~3.5 in Bj\"orn--Bj\"orn~\cite{BBnonopen}, with the obvious
interpretations of the integrals with respect to $\muhat$.
The only
additional observation needed is that if $\Ga$ is a family of curves
in $G$ then by \eqref{eq-int-X-Xhat-equal}
and \eqref{eq-int-X-Xhat-equal2},
\begin{equation}   \label{eq-compare-Mod}
\Mod_{p,G}(\Ga) = \inf_\rho \int_G \rho^p\,d\mu
= \inf_\rho \int_{\Om \cap X} \rho^p \,d\mu
= \inf_{\hat{\rho}} \int_{\Om} \hat{\rho}^p \,d\muhat
= \Mod_{p,\Om}(\Ga), 
\end{equation}
where the infima are taken over all $\rho\in L^p(G,\mu)$,
$\rho\in L^p(\Om \cap X,\mu)$ and $\hat{\rho}\in L^p(\Om,\muhat)$
satisfying for all $\ga\in \Ga$,
\[
\int_\ga \rho\,ds \ge 1 
\quad \text{and} \quad
\int_\ga \hat{\rho}\,ds \ge 1, 
\quad \text{respectively.}
\qedhere
\]
\end{proof}

\begin{lem}  \label{lem-same-zero-cap}
Let $G\subset X$ be $\mu$-measurable and \p-path open in
$\Xhat$, $p\ge1$, and $E\subset G$.
Then $\CpG(E)=0$ if and only if $\CpXhat(E)=0$.
\end{lem}

\begin{proof}
By \cite[Proposition~1.48]{BBbook}, $\CpG(E)=0$ if and only if both
$\mu(E)=0$ and $\Mod_{p,G}(\Ga^G_E)=0$, where $\Ga^G_E$ consists of all
curves $\ga\subset G$ which hit $E$, i.e.\ $\ga^{-1}(E)\ne\emptyset$.
A similar equivalence holds for $\CpXhat(E)=0$ and
\[
\Ga^{\Xhat}_E = \{\ga\subset\Xhat: \ga^{-1}(E)\ne\emptyset\}.
\]
Since $G$ is \p-path open in $\Xhat$, for $\Mod_{p,\Xhat}$-almost all
curves in $\Ga^{\Xhat}_E$, the preimage $\ga^{-1}(G)$ is relatively
open in $[0,l_\ga]$ and nonempty, and thus $\ga$ contains a nonconstant
subcurve $\ga'\in \Ga^G_E$.
Hence, by \cite[Lemma~1.34\,(c)]{BBbook} and \eqref{eq-compare-Mod},
\[
\Mod_{p,\Xhat}(\Ga^{\Xhat}_E) \le \Mod_{p,G}(\Ga^G_E).
\]
The reverse inequality is trivial.
Together with \eqref{eq-mu-iff-muhat=0}, this concludes the proof.
\end{proof}

The following examples show that there is no hope to obtain
Lemma~\ref{lem-same-zero-cap} for $\mu$-measurable sets
that are only \p-path almost open in $\Xhat$.

\begin{example}
Let $X=\R^n$ (unweighted), $p\ge 1$, $E=\{x\in \R^n: |x|=1\}$ and
\[
G =\R^n \setm 
\bigcup_{k=1}^\infty \{x:  |x|=1 \pm 2^{-k}\}.
\]
Then $G$ is 
the union of
an open set and a set of Lebesgue measure zero,
and is thus \p-path almost open for all $p\ge1$, by
Theorem~\ref{thm-path-almost-open-char-intro}.
Moreover,
$\Xhat=\R^n$ and $\CpXhat(E)>0=\CpG(E)$.
Indeed, the characteristic function $\chi_E\in \Np(G)$
is admissible for $\CpG(E)$ and has zero as a \p-weak upper gradient
with respect to $G$.
At the same time, the $(n-1)$-dimensional Hausdorff measure of $E$ is nonzero, and so by
Adams~\cite[(12), p.~122]{Adams88} or
Hakkarainen--Kinnunen~\cite[Theorems~4.3 and~5.1]{HaKi}, 
$\CpXhat(E)>0$ holds for $p=1$ and thus for all $p \ge 1$.
When $p>n$, one can also choose $E=\{0\}$ and 
\[
  G =\R^n \setm 
\bigcup_{k=1}^\infty \{x:  |x|=2^{-k}\}
\quad \text{or} \quad G=\{0\}.
\]
\end{example}

\begin{example}
For $\al>1$, let 
$G=\{x=(x',x_n)\in\R^n: |x'|\le x_n^\al\le1\}$ be the closed cusp
in $X=\Xhat=\R^n$, $n \ge 2$, equipped with the 
measure $d\mu(x)=|x|^\be\,dx$, where $\be>-n$.
Note that $\mu$ is globally doubling and 
supports a global 
$1$-Poincar\'e inequality on $\R^n$,
by Corollary~15.35 in
Heinonen--Kilpel\"ainen--Martio~\cite{HeKiMa} and  Theorem~1 in Bj\"orn~\cite{JBFennAnn}.
Since $G$ is the union of
an open set and a set of Lebesgue measure zero,
it is \p-path almost open for all $p\ge1$, by
Theorem~\ref{thm-path-almost-open-char-intro}. 
Testing with $u_j(x)=\min\{1,-(\log x_n)/j\}$
shows that 
\[
\CpG(\{0\})=0 \quad \text{if $1< p\le  \al (n-1)+\be+1$
or $1=p<\al (n-1)+\be+1$,}
\]
while $\CpXhat(\{0\})>0$ for $p>\max\{n+\be,1\}$, by 
\cite[Example~2.22]{HeKiMa},
and for $p=1 \ge n+\be$, by 
Hakkarainen--Kinnunen~\cite[Theorems~4.3 and~5.1]{HaKi}.
Note that for each $p\ge1$ it is possible to find $\be>-n$
so that $\CpXhat(\{0\})>0=\CpG(\{0\})$.
\end{example}

Recall that for an open set $\Om$ in $X$, we let 
\[
\Omhat=\Xhat \setm \itoverline{X \setm \Om},
\]
where the closure is taken in $\Xhat$. 
This makes $\Omhat$ into 
the largest open set in $\Xhat$ such that
$\Om = \Omhat \cap X$.
Note that $X^\wedge=\Xhat$.
We denote balls with respect to $\Xhat$ by $\Bhat$ or 
$\Bhat(x,r)=\{y\in\Xhat : d(x,y)<r\}$,
and balls with respect to $X$ by $B$.
The inclusion $\Bhat(x,r)\subset B(x,r)^\wedge$ can be strict.

If a function $u\colon\Xhat \to \eR$
has a ($1$-weak) upper gradient $g$ on $\Xhat$,
then clearly $g|_X$ is a ($1$-weak) upper gradient of $u|_X$.
The converse is not true in general, as seen e.g.\ in
$X=\R\setm\Q\subset\R=\Xhat$,
but Theorem~\ref{thm-intro} provides
  a converse under suitable assumptions.

For $p>1$ the result corresponding to 
Theorem~\ref{thm-intro} was obtained in
Bj\"orn--Bj\"orn~\cite[Theorem~4.1]{BBnoncomp},
where the reflexivity of $L^p$ 
was used through the application of \cite[Lemma~6.2]{BBbook}).
We shall now explain how Theorem~\ref{thm-intro}
can be obtained for $p=1$
using the Dunford--Pettis theorem (see e.g.\ 
Ambrosio--Fusco--Pallara~\cite[Theorem~1.38]{AFPbook})
instead of reflexivity.
In both cases, the proof is based on discrete convolutions 
and their gradients, as in 
Koskela~\cite[Proof of Theorem~C]{Koskela} and
Heikkinen--Koskela--Tuominen~\cite{HKT}.

\begin{deff}
Given a measurable set $H\subset X$, a sequence 
$\{g_i\}_{i=1}^\infty$ of functions in $L^1(H)$ is 
\emph{equi-integrable} if the following two conditions are satisfied\textup{:}
\begin{enumerate}
\item 
For any $\eps>0$ there is a measurable set $A\subset H$ with $\mu(A)<\infty$
such that 
\[
\int_{H\setminus A}|g_i|\,d\mu<\eps \quad\textrm{for }i=1,2,\ldots.
\]
\item
For any $\eps>0$ there exists $\delta>0$ such that whenever $D\subset H$ is
measurable and  $\mu(D)<\delta$, then
\[
\int_{D}|g_i|\,d\mu<\eps \quad\textrm{for }i=1,2,\ldots.
\]
\end{enumerate}
\end{deff}

Let $\la \ge 1$ and $\Om\subset X$ be 
an open set such that 
the doubling property 
holds within $\Om$.
For each $k=1,2,\ldots$\,, consider
a Whitney-type covering of $\Om$ by
balls $\{\Bik\}_{i}$ with radii
$\rik\le1/k$ and a subordinate Lipschitz partition of unity 
$\{\phiik\}_{i}$ so that
\begin{enumerate}
\renewcommand{\theenumi}{\textup{(\roman{enumi})}}%
\item \label{c-1}
the balls $\tfrac 15 B_{ik}$ are pairwise disjoint, and
$80\la\Bik\subset \Om$ for all $i$;
\item \label{c-2}
the balls $\{10\la\Bik\}_{i}$ have bounded  overlap
\begin{equation}  \label{eq-overlap}
  \sum_{i} \chi_{10\la\Bik}(x) \le m,
  \quad x\in \Om;
\end{equation}
\item \label{c-3}
if $10\la\Bik \cap 10\la\Bjk \ne\emptyset$ then
$\rjk\le2\rik$;
\item \label{c-4}
  each $\phiik$ is 
a nonnegative $C/\rik$-Lipschitz function vanishing outside $2\Bik$;
\item \label{c-5}
  $\sum_{i}\phiik=1$ in $\Om$.
\end{enumerate}
Here 
$m$ and $C$
are constants depending only on $\lambda$ and the doubling constant
$C_\mu$ of $\mu$ within~$\Om$.

For each fixed $k$, we can construct the covering as follows:
For each $x \in \Om$,
let 
$t_{x}$ be the smallest nonnegative integer such that
\begin{equation}   \label{eq-choose-rx}
    r_{x} :=\frac{2^{-t_{x}}}{k}\le \frac{\dist(x,X \setm \Om)}{80\la}.
\end{equation}
Since $X$ is separable and $\{B(x,r_{x})\}_{x \in \Om}$ covers $\Om$,
we can use the $5B$-covering lemma
(see e.g.\ Heinonen--Koskela--Shanmugalingam--Tyson~\cite[p.~60]{HKSTbook})
to find an at most countable
cover of $\Om$
by balls $\Bik:=B(x_{ik},r_{ik})$, $r_{ik}=r_{x_{ik}}$,
such that the balls $\tfrac15 \Bik$
are pairwise disjoint.
Property~\ref{c-1} is
now easy to verify.
For~\ref{c-3} we have from~\eqref{eq-choose-rx}, when
$10\la\Bik \cap 10\la\Bjk \ne\emptyset$,
\[
80\la r_{jk} \le \dist(x_{jk},X\setm\Om) 
\le \dist(x_{ik},X\setm\Om) + 10\la (r_{jk} + r_{ik}),
\]
so that $7r_{jk}\le \dist(x_{ik},X\setm\Om)/10\la+ r_{ik}$.
From~\eqref{eq-choose-rx} we get
\[
r_{jk} \le \frac1k
\quad \text{and} \quad
\min\biggl\{\frac{1}{k},\frac{\dist(x_{ik},X \setm \Om)}{80\la}\biggr\} < 2 r_{ik}.
\]
Combining these gives
\begin{equation*}   
7r_{jk} 
\le \min\biggl\{\frac{7}{k},\frac{\dist(x_{ik},X\setm\Om)}{10\la}\biggr\}
+ r_{ik} < 17r_{ik},
\end{equation*}
and, by construction,
 the quotient $r_{jk}/r_{ik}$ can only
take dyadic values.

For a fixed $i$, let $J_i=\{j: 10\la\Bik \cap 10\la\Bjk \ne\emptyset\}$.
If $j \in J_i$ then it follows from \ref{c-3} that 
$B_{jk} \subset 40 \la \Bik$.
The ball 
$40 \la \Bik$ is a globally doubling metric space,
by Bj\"orn--Bj\"orn~\cite[Proposition~3.4]{BBsemilocal},
with a doubling constant only depending
on $C_\mu$.
As the balls
$\{B(x_{jk},\tfrac{1}{10} r_{ik})\}_{j \in J_i}$
are pairwise disjoint, property~\ref{c-2} is satisfied
with  $m$ only depending on $\la$ and $C_\mu$.

Finally, a Lipschitz partition of unity satisfying \ref{c-4} and~\ref{c-5} can
now be constructed as in \cite[pp.~104--105]{HKSTbook}.

The following lemma is a special case of
Hakkarainen--Kinnunen--Lahti--Lehtel\"a~\cite[Lemma~4.2 and Remark~4.3]{HKLL}; 
see also Franchi--Haj\l{}asz--Koskela~\cite[Lemma~6]{FHK}
for an earlier very similar result.
The last statement in the lemma is obtained using the
Dunford--Pettis theorem
(see e.g.\ 
Ambrosio--Fusco--Pallara~\cite[Theorem~1.38]{AFPbook}).

\begin{lem}\label{lem:generalized equiintegrability}
Assume that $\mu$ is doubling within an open set $\Omega$ 
and for each $k=1,2,\ldots$\,, let $\{\Bik\}_{i}$
be the above Whitney-type covering of $\Om$.
For $g\in L^1(\Om)$ define the functions
\begin{equation*}  
	g_k:=\sum_{i=1}^{\infty}\chi_{\Bik}\vint_{10\lambda \Bik}g\,d\mu,
        \quad k=1,2,\ldots.
\end{equation*}
Then the sequence $\{g_k\}_{k=1}^\infty$ is equi-integrable.
Moreover, a subsequence of $g_k$ converges weakly in $L^1(\Om)$ 
to a function $\gt$ satisfying $\gt \le mg$ a.e.\  in $\Om$, where
$m$ is as in \eqref{eq-overlap}.
\end{lem}

\begin{proof}[Proof of Theorem~\ref{thm-intro}]
We want to extend $u\in \Done(\Om)$ and its
minimal $1$-weak upper gradient $g_{u}:=g_{u,\Om}$
to $\Om^{\wedge}$.
Consider the above Whitney-type
covering and Lipschitz partition of unity (extended continuously 
to $\Om^{\wedge}$).

As in the proofs of
Heikkinen--Koskela--Tuominen~\cite[Lemma~5.3]{HKT} and
Bj\"orn--Bj\"orn~\cite[Theorem~4.1]{BBnoncomp}, 
it can be shown that for each $k=1,2,\ldots$ and some $C_0$,
depending only on the doubling and Poincar\'e constants within $\Om$,
the constant
functions $C_0\vint_{10\lambda \Bik}g_u\,d\mu$
are upper gradients of 
\[
u_k:= \sum_j u_{\Bjk}\phijk  \quad \text{in }  
\widehat{B}_{ik}:=\widehat{B}(x_{ik},r_{ik}),
\]
where $x_{ik}$ are the centres of $\Bik$. Hence
\[
g_k:=C_0 \sum_{i}\chi_{\widehat{B}_{ik}}
         \vint_{10\lambda \Bik}g_u\,d\mu
\]
is an upper gradient of $u_k$ in $\Om^{\wedge}$.
Moreover, 
by \eqref{eq-overlap} and the doubling property of $\mu$, we have for every
Lebesgue point $x\in\Om$ of $u$ that
\[
|u_k(x)-u(x)| = \biggl| \sum_{2\Bik\ni x} (u_{\Bik}-u(x)) \phiik(x) \biggr|
\le \sum_{2\Bik\ni x} \vint_{\Bik}|u-u(x)|\,d\mu \to 0
\]
as $k\to0$.
Since  
$\mu$ is doubling within $\Om$ and
$u \in L^1\loc(\Om)$ (see Remark~\ref{rem-loc-spcs} below), $u$ 
has Lebesgue points a.e., by e.g.\ Heinonen~\cite[Theorem~1.8]{heinonen}.
We thus conclude that
$u_k\to u$ a.e.\ in $\Om$.

Lemma~\ref{lem:generalized equiintegrability} shows that the sequence
$\{g_{k}\}_{k=1}^\infty\subset L^1(\Om)$ is equi-integrable and 
there exists a subsequence 
(also denoted $\{g_{k}\}_{k=1}^\infty$) converging 
weakly in $L^1(\Om)$ (and hence also in $L^1(\Om^{\wedge})$)
to a function $g$ such that $g\le C_0mg_u$
a.e.\ in $\Om$.

Mazur's lemma, applied repeatedly to the
subsequences $\{g_{k}\}_{k=j}^\infty$, $j=1,2,\ldots$\,,
provides us with convex combinations of $g_k$ converging to $g$
in $L^1(\Om^{\wedge})$ and such that the corresponding convex combinations
of $u_k$ converge a.e.\ 
to the function $\uhat:=\limsup_{k\to\infty}u_k$ on $\Om^{\wedge}$,
which has $g$ as a $1$-weak upper gradient (with respect to
$\Om^{\wedge}$),
see \cite[the proof of Proposition~2.3]{BBbook}.
In particular $\uhat \in D^1(\Om^{\wedge})$.
Since $\uhat=u$ a.e.\ in $\Om$ and $u,\uhat\in D^1(\Om)$, also
$\uhat=u$ $\ConeOm$-q.e.\ in~$\Om$,
and thus $\ConeX$-q.e.\ in~$\Om$ by Lemma~2.24 in \cite{BBbook}.

If $\Om$ is $1$-path open in $\Xhat$ then also the capacities 
$\ConeOm$ and $\ConeXhat$ have the same zero sets in $\Om$,
by Lemma~\ref{lem-same-zero-cap}.
This shows that we may choose $\uhat=u$ in $\Om$.
Lemma~\ref{lem-gu-on-p-path-open-Xhat} 
then shows that
$g_u=g_{\uhat}$ a.e.\
within $\Om$.

Finally, if $\ut$ is defined to be the right-hand side of
	\eqref{eq-Leb-pt-1}, then $\uhat=\ut$ 
	at all Lebesgue points of $\uhat$, i.e.\ $\CpXhat$-q.e.\ in $\Omhat$,
	by the proof of Proposition~\ref{prop-Leb-pt} below  with $\Ghat=\Omhat$.
	Hence,
	$\uhat$ may also be chosen so that it satisfies \eqref{eq-Leb-pt-1}.
\end{proof}

\begin{remark}  \label{rem-various-ass}
(a)
	The simple example $X=\Om=\R \setm \{0\}$ with 
$u(x)=\chi_{(0,\infty)}(x)(1-|x|)_\limplus$ 
	demonstrates that 
	under local assumptions on the measure $\mu$, functions
	in $\Np(\Om)$
	may fail to have extensions even to $\Dploc(\Omhat)$
when $p \ge 1$.
	A partial remedy for this situation is provided by
	Proposition~\ref{prop-ext-Nploc} below.

(b)
	Under semilocal assumptions for $\mu$
	(see Definition~\ref{def-local-intro}),
	the conclusion of Theorem~\ref{thm-intro} clearly
	holds 
	for all bounded $\Om$.
	If $\Om$ is unbounded, 
the semilocal assumptions do not imply 
the doubling property
	and the $1$-Poincar\'e inequality 
	within $\Om$, and so Theorem~\ref{thm-intro} is not directly available.
However, if $\Om$ is $1$-path open in $\Xhat$, then
so 
is $\Om \cap B(x,k)$ for every $k$ and some fixed $x\in\Om$.
Since the doubling property
and the $1$-Poincar\'e inequality hold within each $\Om \cap B(x,k)$, 
applying Theorem~\ref{thm-intro} for each $k$ and letting $k \to \infty$ 
shows that the conclusion holds with $A_0=1$ also for~$\Om$.

(c) Theorem~\ref{thm-intro} is formulated
under assumptions holding within $\Om$. 
The corresponding result \cite[Theorem~4.1]{BBnoncomp} for $p>1$ 
can also be formulated similarly,
with the proof given in \cite{BBnoncomp} still applying.
\end{remark}

\begin{remark}   \label{rem-loc-spcs}
The extension result in
Theorem~\ref{thm-intro}
makes it possible to
obtain quasicontinuity and Lebesgue points for local
Newtonian functions on
noncomplete spaces under local assumptions.
If $X$ supports a local $1$-Poincar\'e inequality then 
$\Noneloc(\Om)=\Doneloc(\Om)$ for every open $\Om \subset X$;
this follows as in \cite[Proposition~4.14]{BBbook}.
Moreover, since local assumptions are
  inherited by open subsets, {the} results
in the rest of this section directly apply also to open $\Om \subset X$.
We therefore formulate them using $\Noneloc(X)$ rather than
$\Noneloc(\Om)=\Doneloc(\Om)$.
\end{remark}

\begin{prop} \label{prop-ext-Nploc}
Assume that 
$\mu$ is locally doubling and supports a local $1$-Poincar\'e 
inequality on $X$.
Then for every $u \in \Noneloc(X)$ there is an open set
$\Ghat \supset X$ in $\Xhat$ 
and a function $\uhat \in \Noneloc(\Ghat)$ such that $u=\uhat$ 
$\ConeX$-q.e.\ on $X$.
Moreover, $\Ghat$ is locally compact and $\mu|_{\Ghat}$ is locally doubling
and supports a local $1$-Poincar\'e inequality. 

If $X$ is $1$-path open in $\Xhat$, then one can choose
$\uhat\equiv u$ 
	and $g_{\uhat} \equiv g_u$  in~$X$.
\end{prop}

Note that the set $\Ghat$ in general depends on $u$, cf.\ 
Bj\"orn--Bj\"orn~\cite[Example~4.7]{BBnoncomp}.

\begin{proof}
Since $X$ is Lindel\"of, we can find a countable cover of $X$
by balls $B_j=B(x_j,r_j)\subset X$
such that $u \in \None(B_j)$ and both the $1$-Poincar\'e inequality and the 
doubling property for $\mu$ hold within each $B_j$, $j=1,2,\ldots$\,.
Let $\Bhat_j=\Bhat(x_j,r_j)$ and $\Ghat=\bigcup_{j=1}^\infty \Bhat_j$.

Using
Theorem~\ref{thm-intro},
we can extend $u|_{B_j}$ to 
$\uhat_j\in\None(\Bhat_j)$ so that $\uhat_j=u$ $\ConeX$-q.e.\ in
$B_j$, $j=1,2,\ldots$\,.
Then $\uhat_i=\uhat_j$ a.e.\ (and hence 
$\ConeXhat$-q.e.) in $\Bhat_i\cap\Bhat_j$ for all $i,j$.
We can thus construct $\uhat \in \Nploc(\Ghat)$ 
so that $\uhat=u$ $\ConeX$-q.e.\ in $X$ and
$g_\uhat \le A_j g_u$ a.e.\ in $B_j$, 
where $A_j$ is the constant provided by
Theorem~\ref{thm-intro}
in $B_j$.
Hence $\uhat\in \Noneloc(\Ghat)$.
If $X$ is $1$-path open in $\Xhat$, then it follows from the last part of 
Theorem~\ref{thm-intro}
that we can choose 
$\uhat\equiv u$ 
	and $g_{\uhat} \equiv g_u$  in~$X$.

The local doubling property and the local $1$-Poincar\'e inequality
for $\mu|_{\Ghat}$ follow from
Bj\"orn--Bj\"orn~\cite[Propositions~3.3 and~3.6]{BBnoncomp}.
Consequently, each $\Bhat_j$ (and thus also $\Ghat$)
is locally compact, by  \cite[Proposition~3.9]{BBnoncomp}.
\end{proof}

The following two results are now relatively easy consequences
of the above extension to $\Ghat \subset \Xhat$
and the corresponding results in complete spaces.
Recall the definition of quasicontinuity from Definition~\ref{deff-q-cont}.

\begin{cor}  \label{cor-qcont}
Assume that  
$\mu$ is locally doubling and supports a local $1$-Poincar\'e 
inequality on $X$, and that
$X$ is $1$-path open in $\Xhat$.
Then every $u \in \Noneloc(X)$ is $\ConeX$-quasicontinuous.
\end{cor}

\begin{proof}
  Find a locally compact open set $\Ghat\subset\Xhat$ and a function
  $\uhat\in \Noneloc(\Ghat)$ as in Proposition~\ref{prop-ext-Nploc}
with $\uhat\equiv u$ in $X$
and so that $\mu|_{\Ghat}$ is locally doubling
and supports a local $1$-Poincar\'e inequality.
It then follows from Theorem~9.1 in Bj\"orn--Bj\"orn~\cite{BBsemilocal}
that $\uhat$ is
$C_1^{\Ghat}$-quasicontinuous
on $\Ghat$, which immediately yields
that $u$ is $\ConeX$-quasicontinuous on $X$, since $\ConeX$ is dominated by
$C_1^{\Ghat}$.
\end{proof}

\begin{prop} \label{prop-Leb-pt}
Assume that $\mu$ is locally doubling and supports a local $1$-Poincar\'e 
inequality on $X$.
Then every $u \in \Noneloc(X)$ has Lebesgue points
$\ConeX$-q.e.,
and moreover the extension $\uhat$ in Proposition~\ref{prop-ext-Nploc}
can be given by
\begin{equation} \label{eq-Leb-pt-2}
    \uhat(x)=\limsup_{r \to 0} \vint_{\Bhat(x,r) \cap X} u \, d\mu,
   \quad x \in \Ghat.
\end{equation}
\end{prop}

The proof below shows that the limit 
\[
\lim_{r \to 0} \vint_{B(x,r)} u \, d\mu
\]
actually exists for $\ConeXhat$-q.e.\ $x\in X$, even though it only equals $u(x)$
for $\ConeX$-q.e.\ $x$.
In general, $\ConeX \le \ConeXhat$, but it follows
  from Lemma~\ref{lem-same-zero-cap}
that they have the same zero sets 
  if $X$ is $1$-path open in $\Xhat$.

\begin{remark}   \label{rem-finite mu}
Even when $X$ is complete, the Lebesgue point result in
Proposition~\ref{prop-Leb-pt} generalizes earlier results 
obtained under global assumptions, as in  
Kinnunen--Korte--Shanmugalingam--Tuominen~\cite[Theorem~4.1 and Remark~4.7]{KKST}.
Therein, $\mu(X)=\infty$ is assumed, but we shall now explain how 
the Lebesgue point result from \cite[Theorem~4.1 and Remark~4.7]{KKST}
can be obtained also for
a complete metric space $X$ equipped with a globally
  doubling measure $\mu$ supporting a global $1$-Poincar\'e inequality
  and satisfying $\mu(X)<\infty$.
  (Under these assumptions,
  $\mu(X)<\infty$ if and only if $X$ is bounded.)
We will use this fact when proving Proposition~\ref{prop-Leb-pt}.

For this, let $\Xt= X\times \R$, equipped with
the product metric
\[
  d_{\Xt}((x,t),(y,s))=\max\{d(x,y),|t-s|\}
\]  and
the product measure
\[
d\mut(x,t) = d\mu(x)\,dt.
\]
Note that $\mut(\Xt)=\infty$.
By Bj\"orn--Bj\"orn~\cite[Theorem~3 and Remark~4]{BBtensor},
  $\mut$ is globally doubling and supports a global $1$-Poincar\'e inequality.
Let $\eta$ be a Lipschitz cut-off function on $\R$ such that $\eta=1$ in $[-1,1]$
and $\eta=0$ outside $[-2,2]$.
If $u \in  \None(X)$ then 
\[
\ut(x,t):=u(x) \eta(t) \in \None(\Xt)
\]
 and 
\cite[Theorem~4.1 and Remark~4.7]{KKST} implies that $\ut$ has Lebesgue 
points at $\ConeXt$-q.e.\ $x\in\Xt$.
Clearly, for $0<r<1$,
\[
\vint_{B(x,r) \times (-r,r)} \ut \, d\mut = 2r \vint_{B(x,r)} u\, d\mu
\quad \text{and}  \quad
\mut(B(x,r) \times (-r,r)) = 2r \mu(B(x,r)),
\]
which implies that $x\in X$ is a Lebesgue point  of $u$ if and only if
$(x,t)\in\Xt$ is a Lebesgue point of $\ut$ for some (and equivalently 
all) $t\in (-1,1)$.
Hence, if $E\subset X$ is the set of non-Lebesgue points of $u$,
then $\ConeXt(E\times (-1,1))=0$ and for every $\eps>0$
there exists $\vt\in \None(\Xt)$, with an upper gradient $g$,
such that $\vt\ge1$ on $E\times (-1,1)$ and 
\[
\int_{\Xt} (|\vt| + g) \,d\mut < 2\eps.
\]
Then there exists $t\in (-1,1)$ such that 
\begin{equation}   \label{eq-intX<eps}
\int_{X} (|v(x,t)| + g(x,t)) \,d\mu(x) < \eps.
\end{equation}
Clearly, $g(\cdot,t)$ is an upper gradient of $v(\cdot,t)$ with respect to $X$
and  $v(\cdot,t)$ is admissible for $\ConeX(E)$.
It therefore follows from~\eqref{eq-intX<eps} that
$\ConeX(E)<\eps$.
Letting $\eps\to0$ now shows that $\ConeX(E)=0$ and so $u$
has Lebesgue points $\ConeX$-q.e.\ in $X$.
\end{remark}

\begin{proof}[Proof of Proposition~\ref{prop-Leb-pt}]
  Find $\Ghat$ and $\uhat\in \Noneloc(\Ghat)$ as in
  Proposition~\ref{prop-ext-Nploc}.
Let $x \in \Ghat$.
As $\Ghat$ is locally compact, it follows from Theorem~\ref{thm-rajala-iter}
that there is a bounded uniform domain $G_x$ in $\Xhat$ 
such that $x \in G_x \Subset \Ghat$ and such that  $\mu|_{\clG_x}$
is globally doubling and supports a global \p-Poincar\'e inequality
on $\clG_x$, where the closure is taken with respect to $\Xhat$. 
In particular, $\uhat \in \None(\clG_x)$.

By~\cite[Theorem~4.1 and Remark~4.7]{KKST}
and the argument in Remark~\ref{rem-finite mu},
$\uhat$ has Lebesgue points $\ConeGx$-q.e.\ in $G_x$.
By Lemma~\ref{lem-same-zero-cap},
the capacities $\ConeGx$ and $\ConeXhat$ have the same zero sets in $G_x$.
Hence as $\Ghat$ is Lindel\"of, $\uhat$ has Lebesgue points 
$\ConeXhat$-q.e.\ in $\Ghat$, and so $u$ has Lebesgue points $\ConeX$-q.e. in $X$.

Finally, if $\ut$ is given by the right-hand side of
\eqref{eq-Leb-pt-2}, then $\uhat=\ut$ 
at all Lebesgue points of $\uhat$, i.e.\ $\CpXhat$-q.e.\ in $\Ghat$. 
Hence,
$\uhat$ may also be chosen so that it satisfies \eqref{eq-Leb-pt-2}.
\end{proof}

Even for $u \in \None(X)$, 
(the proof of)
Proposition~\ref{prop-ext-Nploc} only guarantees an extension in the local 
Newtonian space $\Noneloc(\Ghat)$ (but with $\Ghat$ independent of $u$), 
unless $X$ is 1-path open in $\Xhat$.
However, under slightly stronger uniform assumptions we can obtain
the following partial nonlocal 
conclusion, which also includes $p>1$, see \cite[Remark~4.10]{BBnoncomp}.

\begin{prop} \label{prop-ext-Np}
Assume that there are constants $\Cmu$, $\CPI$ and $\la$ such
that for each $x \in X$, there is $r_x>0$
such that $\mu$ is doubling within $B_x=B(x,r_x)$ with constant $\Cmu$
and $\mu$ supports a \p-Poincar\'e inequality within $B_x$
with constants $\CPI$ and $\la$.

Then  there is an open set $\Ghat \supset X$ in $\Xhat$
such that for every $u \in \Np(X)$, the function $\uhat$
given by~\eqref{eq-Leb-pt-2} satisfies
$\uhat=u$ $\CpX$-q.e.\ on $X$ and belongs to $\Np(\Ghat)$.

If also $r_x$ is independent of $x$ then we may choose
  $\Ghat=\Xhat$.
\end{prop}

Such assumptions are called \emph{semiuniformly local}, 
   and \emph{uniformly local} in the case where $r_x$ is independent of $x$,
in \cite[Definition~6.1]{BBsemilocal}.
  Riemannian manifolds always support at least semiuniformly local assumptions
and often uniformly local ones.
Uniformly local assumptions are natural e.g.\ on Gromov hyperbolic
spaces, see
Bj\"orn--Bj\"orn--Shanmugalingam~\cite{BBSunifPI}, \cite{BBShyptrace}
and Butler~\cite{butler}.
Semiuniformly local assumptions were also used by
e.g.\ Holopainen--Shan\-mu\-ga\-lin\-gam~\cite{HoSh}.

\begin{proof}
Let $\Bhat_x=\Bhat(x,r_x)$ and $\Ghat=\bigcup_{x \in X}\Bhat_x$.
By \cite[Proposition~4.8 and the proof of Lemma~4.6]{BBnoncomp}
(for $p >1$)
or
Proposition~\ref{prop-Leb-pt} and the proof of 
Proposition~\ref{prop-ext-Nploc}  (for $p=1$),
we get that $\uhat \in \Nploc(\Ghat)$.
By
\cite[Theorem~4.1]{BBnoncomp} (for $p >1$)
and Theorem~\ref{thm-intro} (for $p=1$),
we see that $g_\uhat \le A_0 g_u$ a.e.\ in $X$, where
$A_0$ only depends on $p$, $\Cmu$, $\CPI$ and $\la$.
Thus
\[
\int_{\Ghat} |\uhat|^p \, d\muhat= \int_{X} |u|^p \, d\mu < \infty
\quad \text{and} \quad
\int_{\Ghat} g_\uhat^p \, d\muhat \le  A_0^p\int_{X} g_u^p \, d\mu < \infty,
\]
i.e.\ $\uhat \in \Np(\Ghat)$.
If $r_x$ is independent of $x$, then clearly $\Ghat=\Xhat$.
\end{proof}

\section{Removable sets for Newtonian spaces}
\label{sect-rem}

\emph{We assume in this section that $1\le p<\infty$ and that
$Y=(Y,d,\mu_Y)$ is a metric measure space equipped
with a metric $d$ and a positive complete  Borel  measure $\mu_Y$
such that $0<\mu_Y(B)<\infty$ for all 
balls $B \subset Y$.
Moreover, $X\subset Y$ is such that $Y \subset \Xhat$. 
We also let $E=Y \setm X$ and assume that the inner measure satisfies
\begin{equation}   \label{eq-mu-Y,in}
\muYin(E)
:= \sup \{\mu_Y(A): A \subset E \text{ is $\mu_Y$-measurable}\}=0.
\end{equation}
}

Our main interest in this section is
removability of sets with zero  measure,
i.e.\ when $X \subset Y$  are two metric spaces with 
$\mu_Y(Y \setm X)=0$.
In order to be able (as before) to include the case when $Y=\Xhat$ and $X$ is a
nonmeasurable subset of $Y$, we merely impose the condition
\eqref{eq-mu-Y,in}.
This will only necessitate a little extra care in some of the formulations.
At the end of this section we give examples of nonmeasurable
 removable sets with zero inner measure.
Removability of sets with positive measure
is a different topic, related to extension domains,
see e.g.\ Haj\l asz--Koskela--Tuominen~\cite{HaKoTuo} and
Bj\"orn--Shan\-mu\-ga\-lin\-gam~\cite[Section~5]{BjShJMAA}.
As in \eqref{eq-muhat-Borel}, it follows
  that 
\[
\muYin(E)
= \sup \{\mu_Y(A): A \subset E \text{ is a Borel set in $Y$}\}.
\]

Since we want $Y$ to satisfy our standing assumption that balls
have positive measure, necessarily
$Y \subset \Xhat =\Yhat$.
In the nonmeasurable case, 
we cannot just let $\muX=\muY|_X$, but need to define
$\muX$ by letting
\begin{equation} \label{eq-muY}
    \mu_X(A \cap X)=\mu_Y(A) \quad 
   \text{for every } \muY \text{-measurable set }
A\subset Y.
\end{equation}
This
is well-defined since $\muYin(E)=0$, and
makes $\mu_X$ into a complete
Borel regular measure on $X$,
which
coincides with the restriction  $\mu_Y|_X$ when $X$ is $\muY$-measurable.

We note that q.e.\ defined equivalence classes
may depend on whether the capacity is $\CpX$ or $\CpY$,
whereas the a.e.\ condition coincides in both spaces,
due to \eqref{eq-mu-iff-muhat=0}.
So
for simplicity
we restrict the discussion to removability with respect to the
following spaces,
where we
implicitly 
	assume that $u\colon X\to\eR$ is
  defined pointwise in~$X$:
\begin{align*}
        \hNp (X) &= \{u : u=v \text{ a.e. for some } v \in \Np(X)\},\\
        \hDp (X) &= \{u : u=v \text{ a.e. for some } v \in \Dp(X)\}.
\end{align*}
In both cases we define $g_u=g_v$.
  This is well-defined a.e.\ and independent of the choice of
 $v$ such that $v=u$ a.e.
The spaces $\hNp(Y)$ and $\hDp(Y)$ are defined similarly.

\begin{deff}
 The set $E=Y \setm X$ is \emph{removable} for $\hNp(X)$ if $\hNp(X)=\hNp(Y)$
 in the sense that $\hNp(X)=\{u|_{X} : u \in \hNp(Y)\}$.
  If moreover, $g_{u,X}=g_{u,Y}$ a.e.\ in $X$ for every  $u \in \hNp(Y)$,
    then $E$ is \emph{isometrically removable} for $\hNp(X)$.

    Removability and isometric removability
  for $\hDp(X)$ are defined similarly.
\end{deff}

It is easily seen that removability for $\hNp$ is the same as for
the corresponding spaces of a.e.-equivalence classes
\begin{equation} \label{eq-quotient}
\hNp(X)/\!\simae
\quad \text{and} \quad  \hNp(Y)/\!\simae,
\end{equation}
where $u \simae v$ if $u-v=0$ a.e.
  However to make it clearer what exactly is meant, especially
  in the nonmeasurable case,
we prefer to work with the spaces $\hNp$ of pointwise defined functions.
In fact, the proofs below show that when $E$ is removable then
  any $\muY$-measurable extension of $u$ from $X$ to $Y$ will do the job.

Note also that the quotient spaces
in \eqref{eq-quotient}
are Banach spaces.
Since clearly $\|u\|_{\Np(X)} \le \|u\|_{\Np(Y)}$,
the bounded 
inverse theorem shows that the norms in these spaces are
equivalent when $E$ is removable for $\hNp(X)$.

As a first observation we deduce the following result.

\begin{prop} \label{prop-rem-Cp=0}
If $\CpY(E)=0$,
then 
$E$ is isometrically removable for $\hNp(X)$ and $\hDp(X)$.
\end{prop}

Note that no assumptions on $Y$ are needed 
(other than the standing assumptions from
the beginning of this section)
and that $X$ is
automatically measurable in this case,
since $\muY(E)=0$ (which follows directly from Definition~\ref{deff-Sob-cap}).

\begin{proof}
Let $\uhat \in \hDp(X)$ and let $u \in \Dp(X)$ be such that $u=\uhat$ a.e.\ in $X$.
Let $g$ be any \p-weak
upper gradient of $u$ in $X$. 
Extend $u$ and $g$ 
by $0$ to $Y \setm X$.
Note that as $X$ is measurable so are $u$ and $g$.
Since $\CpY(E)=0$, it follows from \cite[Proposition~1.48]{BBbook}
that \p-almost no curve in $Y$ hits $E$.
Hence $g$ is a \p-weak upper gradient of $u$ also on $Y$.
Since $u=\uhat$ a.e.\ in $X$, any 
extension
	of
	$\uhat$ to $Y$
	will coincide with $u$ a.e.\ in $Y$ and so
	belongs to $\hDp(Y)$.
Thus $E$ is isometrically removable both for $\hNp(X)$ and $\hDp(X)$.
\end{proof}

\begin{example} \label{ex-Rn-countable}
Let $Y=\R^n$, $n \ge 2$,  $1 \le p \le n$ and
let $E\subset \R^n$ be a countable or finite set.
Then it is 
well known that $C_{p}^{\R^n}(E)=0$, and thus
$E$ is isometrically removable for $\hNp(\R^n\setm E)$ and $\hDp(\R^n\setm E)$,
by Proposition~\ref{prop-rem-Cp=0}.

If $E \subset H$ is dense in a hyperplane $H$, 
then $\clE=H$ is not removable for
$\hNp(\R^n \setm \clE)$ nor for $\hDp(\R^n \setm \clE)$.
This follows from Theorem~\ref{thm-hN1-rem-global} below
since $\R^n \setm H$ is disconnected
and hence
does 
not support any global Poincar\'e inequality.

This shows that 
removability for nonclosed sets
cannot be achieved by only studying removability of their closures.
In Proposition~\ref{prop-Hnone} we give a much more
general result which includes this example as a special case.
\end{example}

The following is the main result in this section.

\begin{thm} \label{thm-hN1-rem-global} 
    Assume that $\muY$  is globally doubling
  and supports a global \p-Poincar\'e inequality on $Y$.
Consider the following statements\/\textup{:}
\begin{enumerate}
\item \label{F-None}
$E$ is removable for $\hNp(X)$.
\item \label{F-Done}
$E$ is removable for $\hDp(X)$.
\item \label{F-None-i}
$E$ is isometrically removable for $\hNp(X)$.
\item \label{F-Done-i}
$E$ is isometrically removable for $\hDp(X)$.
\item \label{F-X-const}
$X$ supports a
  global \p-Poincar\'e inequality with the same
  $C$ and $\la$ as on $Y$.
\item \label{F-X}
$X$ supports a
  global \p-Poincar\'e inequality.
\end{enumerate}
Then \ref{F-None-i} $\eqv$ \ref{F-Done-i} $\imp$
\ref{F-X-const} $\imp$ 
\ref{F-X} $\imp$ 
\ref{F-Done} $\imp$ \ref{F-None}.

If in addition $X$ is \p-path almost open in $Y$,
then \ref{F-None}--\ref{F-X} are all equivalent.
\end{thm}

As mentioned in the introduction, this
generalizes Theorem~C in Koskela~\cite{Koskela},
  see also Koskela--Shanmugalingam--Tuominen~\cite[p.~335]{KoShTu00}.
Koskela
obtained such a characterization of removability
for $W^{1,p}(\R^n\setm E)$ on unweighted $\R^n$, with $p>1$
and $E$
closed
(and thus $X=\R^n \setm E$ open and hence \p-path almost open).
In the classical situation, on unweighted $\R^n$, our result
thus extends Koskela's result to $p=1$.
The classical Sobolev spaces
$W^{1,p}(\R^n)$
and $W^{1,p}(\R^n \setm E)$, for
$E$ closed,
coincide with $\hNp(\R^n)/{\simae}$ and $\hNp(\R^n \setm E)/{\simae}$
(with the same norm),
respectively,
by Theorem~7.13 in Haj\l asz~\cite{Haj03}
(or \cite[Corollary~A.4]{BBbook}).
This is true also in weighted Euclidean spaces,
for \p-admissible weights when $p>1$, see \cite[Proposition~A.12]{BBbook}.
(A weight $w$ is \emph{\p-admissible} if $d\mu=w\,dx$
is a globally doubling measure supporting a global
\p-Poincar\'e inequality.)
For $p=1$ and a $1$-admissible weight,
Proposition~4.26 in Cheeger~\cite{Cheeg},
together with the arguments in \cite[Propositions~A.11 and~A.12]{BBbook},
implies that the norms are comparable,
see also Eriksson-Bique--Soultanis~\cite{SEB-Soultanis}.

Theorem~\ref{thm-A} below shows 
that the assumptions in Theorem~\ref{thm-hN1-rem-global}
can be fulfilled without $X$ being \p-path almost open in $Y=\Xhat$,
and that even in this case it is possible 
that \ref{F-None}--\ref{F-X} all hold.
Some of the implications hold under weaker assumptions and
we begin with deducing them.

\begin{prop} \label{prop-hDp=>hNp}
If $E$ is removable for $\hDp(X)$,
then it is removable for $\hNp(X)$.
\end{prop}

\begin{proof}
    Let $u \in \hNp(X)$.
  Since $u \in \hDp(X)$ and $E$ is removable for $\hDp(X)$,
there exists
  $\uhat \in \hDp(Y)$ such that $\uhat=u$ in $X$.
As
$\|\uhat\|_{L^p(Y)}= \|u\|_{L^p(X)}<\infty$ 
by~\eqref{eq-muY},
we see that
$\uhat \in \hNp(Y)$.
Hence $E$ is removable for $\hNp(X)$.
\end{proof}

\begin{prop} \label{prop-PI=>rem}
 Assume that $\muX$ is
 doubling
 and supports a
 \p-Poincar\'e inequality within an open
   set $\Om\subset X$.
 Then $E \cap \Omhat$ is removable both for $\hNp(\Om)$ and $\hDp(\Om)$.
\end{prop}

\begin{proof}
  By Proposition~\ref{prop-hDp=>hNp}
(with $\Om$ in place of $X$),
 it suffices to prove removability for
  $\hDp(\Om)$.
  Let $\uhat \in \hDp(\Om)$.
Then there is $u \in \Dp(\Om)$ such that $u=\uhat$ a.e.\ in $\Om$.
By
Theorem~\ref{thm-intro}
(when $p=1$)
  and \cite[Theorem~4.1]{BBnoncomp} (when $p>1$,
see Remark~\ref{rem-various-ass}\,(c)), 
there exists
$v \in \Dp(\Omhat)$ such that $v=u$ $\CpX$-q.e.\ in $\Om$.
Since $v=\uhat$ a.e.\ in $\Om$, any $\muY$-measurable extension
 of
 $\uhat$ to $Y \cap \Omhat$
 will coincide with $v$ a.e.\ in $Y \cap \Omhat$ and so
  belongs to $\hDp(Y \cap \Omhat)$.
Hence $E \cap \Omhat$ is removable for $\hDp(\Om)$.
\end{proof}

\begin{thm} \label{thm-isom-rem}
  The set $E$ is isometrically removable for $\hNp(X)$
  if and only if it is isometrically removable for $\hDp(X)$.
\end{thm}

\begin{proof}
Assume first that $E$ is isometrically removable for $\hDp(X)$.
By Proposition~\ref{prop-hDp=>hNp}, the set $E$ is removable for
  $\hNp(X)$.
As the removability for $\hDp(X)$ is isometric, it follows
directly from the definition that $E$ is
isometrically removable also for $\hNp(X)$.

Conversely, assume that $E$ is isometrically removable for $\hNp(X)$.
Let $u \in \hDp(X)$ and let $v \in \Dp(X)$ be such that $v=u$ a.e.\ in $X$.
First consider the case when  $u \ge 0$, so that we can assume also $v\ge 0$.
 Fix $x_0 \in X$
and let 
\[
    u_k(x)=(1-\dist(x,B_X(x_0,k))_\limplus \min\{v(x),k\},
    \quad k=1,2,\ldots.
\]
Then $u_k \in \Np(X)$ and there is
$\uhat_k \in \Np(Y)$ such that $\uhat_k=u_k$ a.e.\ in $X$, and thus
$\CpX$-q.e.\ in $X$.
As $\uhat_{k+1} \ge \uhat_k$ a.e., it follows
from Corollary~1.60 in \cite{BBbook}
that  $\uhat_{k+1} \ge \uhat_k$ $\CpY$-q.e.,
and thus we can choose $\uhat_{k+1}$ so that $\uhat_{k+1} \ge \uhat_k$
everywhere.
Hence $\uhat=\lim_{k \to \infty} \uhat_k$
is well-defined pointwise.

Next, let $\ghat=g_{u,X}$, extended measurably
to $Y \setm X$.
By the isometric removability  and truncation,
$g_{\uhat_k,Y} = g_{u_k,X} \le  \ghat$ a.e.\ in $B_Y(x_0,k)$,
and thus 
$ \ghat$ is 
a \p-weak upper gradient of $\uhat_k$
in $B_Y(x_0,k)$.
Since by~\eqref{eq-muY},
\[
 \mu_Y(\{x \in Y : |\uhat(x)|=\infty\})
= \mu_X(\{x \in X : |v(x)|=\infty\})
=0, 
\]
it follows from
Lemma~1.52 in \cite{BBbook} that $\ghat$ is a \p-weak upper gradient of $\uhat$
in each $B_Y(x_0,k)$ and hence in $Y$.
Therefore $\uhat \in \Dp(Y)$, and clearly $\uhat=u$ a.e.\ in $X$.
Now any $\muY$-measurable extension of $u$ will
coincide with $\uhat$ a.e.\ in $Y$ and so belongs to $\hDp(Y)$.
For general $u$ we
write $u=u_\limplus - u_\limminus$,
extend $u_\limplus$ and $u_\limminus$
as above, and  take their difference.
Thus $E$ is isometrically removable for $\hDp(X)$.
\end{proof}  

\begin{remark} \label{rmk-Dp}
In the proof of Theorem~\ref{thm-isom-rem}, we used the fact 
that functions in $\Dp(X)$ are finite a.e.\
when applying  \cite[Lemma~1.52]{BBbook}.
This is the reason why 
our definition of $\Dp(X)$ slightly deviates from the one
in \cite{BBbook}, see Section~\ref{sect-ug}.
It may also be more natural to only consider functions
that are finite a.e. 
\end{remark}

\begin{proof}[Proof of Theorem~\ref{thm-hN1-rem-global}]
\ref{F-None-i} $\eqv$ \ref{F-Done-i}  
This follows from Theorem~\ref{thm-isom-rem}.

  \ref{F-None-i} $\imp$ \ref{F-X-const}
  Let $u \in \Np(X) \subset \hNp(X)$. 
As $E$ is isometrically removable for $\hNp(X)$, 
 there is $\uhat \in \Np(Y)$ such that $\uhat=u$ a.e.\ in $X$
and $g_{\uhat,Y}=g_{\uhat,X}=g_{u,X}$ a.e.\ in $X$,
see Section~\ref{sect-ug}.
Let $B_X=B_X(x,r)$ be a ball in $X$, and $B_Y=B_Y(x,r)$ be the corresponding
  ball in $Y$. Then 
in view 
of~\eqref{eq-muY}
and using the Poincar\'e inequality on $Y$,
  \begin{align*}
    \vint_{B_X} |u-u_{B_X}|\, d\muX
    &= \vint_{B_Y} |\uhat-\uhat_{B_Y}|\, d\muY\\
    &\le C r  \biggl(\vint_{\la B_Y} g_{\uhat,Y}^p \, d\muY\biggr)^{1/p}
     =  C r  \biggl(\vint_{\la B_X} g_{u,X}^p \, d\muX\biggr)^{1/p}.
  \end{align*}
Thus 
$X$ supports a global \p-Poincar\'e inequality with
  the same constants $C$ and $\la$ as on $Y$,
  by Lemma~\ref{lem-PI-char}.

\ref{F-X-const} $\imp$ \ref{F-X}
  This is trivial.

\ref{F-X} $\imp$ \ref{F-Done}
It follows directly from \eqref{eq-muY} that
  $\muX$ is globally doubling
  on $X$. 
Hence this implication follows from Proposition~\ref{prop-PI=>rem}.

\ref{F-Done} $\imp$ \ref{F-None}
This follows from Proposition~\ref{prop-hDp=>hNp}.

Finally, if $X$ is \p-path almost open in $Y$,
then \ref{F-None} $\imp$ \ref{F-None-i} by
Lemma~\ref{lem-gu-on-p-path-open-Xhat}.
\end{proof}

Under local assumptions we obtain the following result.
Recall that local assumptions are inherited by open sets and
  thus $X$ and $Y$ in the following theorem can be replaced
  by $\Om\cap X$ and $\Om$, respectively, for any open set $\Om\subset Y$,
  cf.~Remark~\ref{rem-loc-spcs}.

\begin{thm}
\label{thm-hN1-rem-local}
  \textup{(Local version)}
    Assume that $\muY$  is locally doubling
    and supports a local \p-Poincar\'e inequality on $Y$.
Consider the following statements\/\textup{:}
\begin{enumerate}
\item \label{E-rem}
$E$ is removable for $\hNp(X)$.
\item \label{E-rem-Done}
$E$ is removable for $\hDp(X)$.
\item \label{E-None-i}
$E$ is isometrically removable for $\hNp(X)$.
\item \label{E-Done-i}
$E$ is isometrically removable for $\hDp(X)$.
\item \label{E-Y}
  Whenever $x \in X$ and
  the Poincar\'e inequality \eqref{eq-PI-on-B} holds for
  a ball $B_Y(x,r)$ in $Y$, 
  it
holds for
the ball $B_X(x,r)$ in $X$  with the same constants $C$ and~$\la$.
\item \label{E-X}
There is a cover of $Y$ by at most countably many balls 
$B_{Y,j}=B_Y(x_j,r_j)$, $x_j\in X$, such that
the \p-Poincar\'e inequality holds within each ball
$B_{X,j}=B_X(x_j,r_j)$. 
\end{enumerate}
Then
\[
  \textstyle 
  \ref{F-None} \revimp \ref{F-Done} \revimp
  \ref{F-None-i} \eqv \ref{F-Done-i} \imp
  \ref{F-X-const} \imp
  \ref{F-X}.
\]
If in addition $X$ is \p-path almost open in $Y$,
then \ref{F-None}--\ref{F-X} are all equivalent.
\end{thm}

Note that $Y=\R$ with $E=\{0\}$ shows that
in order for the 
equivalences in the
last part to hold it is not
  possible to replace \ref{F-X} by the assumption that
  ``$X$ supports a local \p-Poincar\'e inequality''.

\begin{proof}
\ref{E-None-i} $\eqv$ \ref{E-Done-i}  
		This follows from Theorem~\ref{thm-isom-rem}.

\ref{F-Done-i} $\imp$ \ref{F-Done}
This is trivial.

\ref{F-Done} $\imp$ \ref{F-None}
This follows from Proposition~\ref{prop-hDp=>hNp}.

\ref{E-None-i} $\imp$ \ref{E-Y}
The proof of this implication is similar to the proof of
the corresponding implication in
Theorem~\ref{thm-hN1-rem-global}.

\ref{E-Y} $\imp$ \ref{E-X}
Since $Y$ is Lindel\"of and supports a local \p-Poincar\'e inequality, 
this is straightforward.

Now assume that $X$ is \p-path almost open in $Y$.

\ref{F-None} $\imp$ \ref{F-None-i} This follows from
  Lemma~\ref{lem-gu-on-p-path-open-Xhat}.

\ref{E-X} $\imp$ \ref{E-rem-Done}
Since $\muY$ is locally doubling on $Y$, we may assume
that the cover $B_{Y,j}$ has been chosen so that
$\muY$ is doubling within each $B_{Y,j}$.
It follows directly from \eqref{eq-muY}
that $\muX$ is doubling within each $B_{X,j}$.
Let $\uhat \in \hDp(X)$.
Then there is $u \in \Dp(X)$ such that $u=\uhat$ a.e.\ in $X$.
Note that $Y \subset \Xhat$.  

Using
Theorem~\ref{thm-intro}
(when $p=1$)
    and \cite[Theorem~4.1]{BBnoncomp} (when $p>1$),
  we can find $u_j \in \Dp(B_{Y,j})$ 
such that $u_j=u$ $\CpX$-q.e.\ in $B_{X,j}$, $j=1,2,\ldots$\,.
As $u_i,u_j \in D^p(B_{Y,i} \cap B_{Y,j})$ and
the set $\{y \in B_{Y,i} \cap B_{Y,j} : u_i(y) \ne u_j(y)\}$
has measure zero, it must be of zero $\CpY$-capacity for all $i,j$.
We can thus construct $v\in \Dploc(Y)$
such that $v=u_j$ $\CpY$-q.e.\ in $B_{Y,j}$, $j=1,2,\ldots$\,,
and hence $v=u$ $\CpX$-q.e.\ in $X$.

Since $X$ is \p-path almost open in $Y$, we have $g_{u_j,Y}=g_{u,X}$
a.e.\ in $B_{X,j}$,
by Lemma~\ref{lem-gu-on-p-path-open-Xhat}.
As every curve $\ga$ in $Y$ is compact, it
can be covered by finitely many $B_{Y,j}$. 
From this it follows that
$g_{u,X}$ (extended measurably to $Y \setm X$)
is a \p-weak upper gradient also of $v$ in $Y$,
and thus $v \in \Dp(Y)$.
Since $v=\uhat$ a.e.\ in $X$, any $\muY$-measurable extension of
  $\uhat$ will belong to $\hDp(Y)$.
Hence $E$ is removable for $\hDp(X)$.
\end{proof}

The following result, albeit a bit trivial, gives us plenty of examples
of nonmeasurable removable sets with zero inner measure.
Consider e.g.\ $Y$ to be the von Koch snowflake curve
(see e.g.\ \cite[Example~1.23]{BBbook}) and $X \subset Y$ be any nonmeasurable
subset with full outer measure.

\begin{prop} \label{prop-no-curves}
  Assume that there are  no or \p-almost no curves in $Y$,
  i.e.\ that $\Mod_{p,Y}(\Ga)=0$, where $\Ga$ is the collection
  of all nonconstant rectifiable curves in $Y$.
  Then any $E\subset Y$ satisfying \eqref{eq-mu-Y,in} 
is isometrically removable for $\hNp(X)$ and $\hDp(X)$.
\end{prop}

\begin{proof}
  In this case $g_u=0$ a.e.\ for every measurable function
  $u$ on $X$ or $Y$,
and so 
$\hNp(X)=\Np(X)=L^p(X)$ and $\hNp(Y)=\Np(Y)=L^p(Y)$.
It thus follows directly 
from~\eqref{eq-muY}
that $E$ is removable for $\hNp(X)$.
Since
$g_{u,X}=g_{u,Y}$ a.e.\ in $X$,
the removability is isometric.
By Theorem~\ref{thm-isom-rem},
$E$ is isometrically removable also for $\hDp(X)$.
\end{proof}

\section{Extension from a non-\texorpdfstring{\p}{p}-path almost open set}

\label{sec:counterexample}

We are now going to construct a set $X \subset \R^2$
which satisfies the assumptions in
Theorem~\ref{thm-intro}
but
is not \p-path almost open in $\R^2$.
However,
its complement 
  is isometrically removable for $\hNp(X)$ and $\hDp(X)$.

We first construct a planar Cantor set $C\subset [0,1] \times [0,1]$
as follows.
Let $H_0=[0,1]$ and for every $k=0,1,\ldots$\,, let $H_{k+1}$
be the set obtained by removing from the centre of
every interval in $H_k$ the open interval of length $2^{-2k-1}$.
Then let $C=\bigcap_{k=1}^{\infty}(H_k \times H_k)$
which is a planar Cantor set.
This set projects (orthogonally)
onto full intervals on the lines
\begin{equation} \label{eq-lines}
y= \pm \tfrac12 x + c
\quad \text{and} \quad
y=\pm 2x+c,
\end{equation}
but has zero length projections on all other
lines. 
This is easy to check by sketching the set $H_1\times H_1$ and then
	noting the self-similarity of the construction.
In particular, $C$
has $1$-dimensional
Hausdorff measure $0 <\Hone(C)< \infty$
(where the latter inequality is easy to show).

The Cantor set $C$ is often called the
four corners Cantor set, as well as the 
Garnett--Ivanov set
in complex analysis, since
Garnett~\cite{garnett70} and 
Ivanov~\cite[footnote on p.\ 346]{ivanov75}
(independently) showed that it is removable for bounded analytic functions.\footnote{
For the historically interested reader it may be worth noting that
Veltmann~\cite{veltmannA} considered planar Cantor sets in 1882
before Cantor~\cite[p.~590 (p.~407 in \emph{Acta Math.})]{cantor}
published his ternary set in 1883.}

Let next $\{q_j\}_{j=1}^{\infty}$ be an enumeration of $\Q^2$ and define
\begin{equation}    \label{eq-def-A}
A=\bigcup_{j=1}^{\infty} (q_j+C),
\end{equation}
i.e.\ we shift $C$
by all rational numbers and take the union.
We are now going to show the following properties
for $X=\R^2 \setm A$.

\begin{thm} \label{thm-A}
  Let $X = \R^2 \setm A$, where $A \subset \R^2$ is as 
in \eqref{eq-def-A}.
Also let
$\Om'$ be a nonempty open subset of $\R^2$
and $\Om = \Om' \cap X$, all sets being equipped with
the Lebesgue measure $\mathcal L^2$.
Then the following are true\/\textup{:}
\begin{enumerate}
\item \label{k-isom}
  $A \cap \Om'$ is isometrically removable
  for $\hNp(\Om)$\/\textup{;}
\item \label{k-PI}
  $X$ supports a global $1$-Poincar\'e inequality\/\textup{;}
\item \label{k-pathalmostopen}
  $\Om$ is not \p-path almost open in $\R^2$.
\end{enumerate}
\end{thm}

This in particular shows that the assumptions and
conclusions in 
Theorem~\ref{thm-intro},
as well as in the corresponding  Theorem~4.1
in Bj\"orn--Bj\"orn~\cite{BBnoncomp} for $1<p<\infty$,
can be fulfilled even if
$\Om$
is not \p-path almost open in $\Xhat$.
Similarly, it shows that the assumptions in Theorem~\ref{thm-hN1-rem-global}
can be fulfilled without $X$
being \p-path almost open in $Y=\Xhat$,
and that even in this case it is possible 
that \ref{F-None}--\ref{F-X}
all hold.
Moreover, the conclusions in Theorem~\ref{thm-A-intro} hold.

\begin{proof}
  \ref{k-isom}
Let  $u\in \Np(\Om)$.
Then $u$ is absolutely continuous on \p-almost every curve
in $\Om$, by Proposition~3.1 in Shan\-mu\-ga\-lin\-gam~\cite{Sh-rev}
(or \cite[Theorem~1.56]{BBbook}).
Let  $l$ be any line which is \emph{not} among those 
in \eqref{eq-lines}.
The orthogonal
projection of $C$,  and thus of $A$, on $l$ has zero length.
Hence almost every line in $\R^2$, which is perpendicular to $l$,
does not intersect $A$. Thus,
by~\cite[Lemmas~2.14 and~A.1]{BBbook}, 
$u$ is absolutely continuous along the intersection of almost
  every such line with $\Om'$ and the corresponding directional
derivative $u'_d$ of $u$ satisfies $|u'_d| \le g_{u,\Om}$ a.e.
(Note that
$\mathcal L^2(\Om'\setm\Om)=0$.)

In particular, $u \in \ACL(\Om')$ and thus
$u \in W^{1,p}(\Om') = \hNp(\Om')/{\simae}$, by
e.g.\ Theorem~2.1.4 in Ziemer~\cite{Ziemer}.
Since we have only excluded four directions of lines for $l$,
the distributional gradient of $u$ satisfies
$|\nabla u| \le g_{u,\Om}$ a.e.\ in $\Om'$.
Thus,
\[
g_{u,\Om'}=|\nabla u| \le 
g_{u,\Om}
\quad \text{a.e.\ in $\Om'$,}
\] 
by Theorem~7.13 in Haj\l asz~\cite{Haj03} (or \cite[Corollary~A.4]{BBbook}),
while
the reverse inequality is trivial.
Hence $A$ is isometrically removable for $\hNp(\Om)$.

\ref{k-PI}  
This now follows directly from \ref{k-isom} and Theorem~\ref{thm-hN1-rem-global}.

\ref{k-pathalmostopen}
	Consider the family $\Gamma_0$ of all lines
	$\{\gamma(t):=(t/\sqrt{5},(2t+c)/\sqrt{5}) : t\in\R\}_{c\in\R}$.
The crucial property of the Cantor set $C$ is that if
any such line intersects $[0,4^{-i}]\times [0,4^{-i}]$, then it intersects
$4^{-i}C \subset C$, $i=0,1,\ldots$\,, though only in a set of zero $1$-dimensional
        Lebesgue measure.
	Thus if any line $\gamma\in \Gamma_0$ intersects
	$q_j+[0,4^{-i}]\times [0,4^{-i}]$ for some indices $i,j$,
	then it intersects $q_j+4^{-i} C$.
	
	Fix $\gamma\in\Gamma_0$, $t\in \R$ and $\eps>0$.
	We then find $i,j$ such that $4^{-i}<\eps/2$
	and $\gamma(t)\in q_j+[0,4^{-i}]\times [0,4^{-i}]$.
	As explained above, the line $\gamma$ intersects $q_j+4^{-i} C$
	and so there is $s\in \R$ with $|s-t|<\eps$ such that
	$\gamma(s)\in q_j+4^{-i} C\subset A$.
	It follows that
	$\gamma^{-1}(A)$ is dense in $\R$ but of zero $1$-dimensional
        Lebesgue measure.
	The lines $\gamma\in\Gamma_0$ are not rectifiable curves
	since they are not of finite length, but
	we can define $\Gamma$ as the collection of all
	compact line segments on these lines
        that also belong to $\Om'$.
        Let $\ga \colon [0,l_{\gamma}] \to \Om'$,
        $\ga \in \Gamma$,
          be an arc-length parameterized curve.
        Then by the above argument,
	$\gamma^{-1}(A)$ is dense in
	$[0,l_{\gamma}]$ but of zero $1$-dimensional Lebesgue measure,
	and so $\gamma^{-1}(\Om)=[0,l_{\gamma}] \setm \gamma^{-1}(A)$
        is not the union of an
        open set and a set of
	zero $1$-dimensional Lebesgue measure.
	By \cite[Lemma~A.1]{BBbook}, we also have $\Mod_{p,\R^2}(\Ga)>0$
	for all $1\le p<\infty$.
	In conclusion, $\Om$	is not \p-path almost open in $\R^2$.
\end{proof}

In the rest of this section, we
provide examples of removable sets $E$ fulfilling the 
assumptions in Theorem~\ref{thm-hN1-rem-global},
with $X=Y\setm E$ that is \p-path almost open but not \p-path open.

\begin{example} \label{ex-bow-tie-p>2}
Let $p > 2$ and let $Y$ be the so-called bow-tie
\begin{align*}
   Y &=\{(x_1,x_2) \in \R^2 : x_1 x_2 \ge 0\}, \\
   E & =\{(x_1,x_2) \in \R^2 : x_1=0 \text{ or } x_2 =0\} \setm \{(0,0)\}, \\
   X & = Y \setm E.
\end{align*}
We equip $Y$ with the Lebesgue measure, which is globally doubling on $Y$.
Then $Y$ supports a global \p-Poincar\'e inequality, by
\cite[Example~A.23]{BBbook}.
The same proof also shows that $X$ supports a global \p-Poincar\'e inequality.
By Theorem~\ref{thm-path-almost-open-char-intro}, $X$ is \p-path almost
open in $Y$.
Thus, by Theorem~\ref{thm-hN1-rem-global}, $E$ is isometrically removable
for $\hNp(X)$.
Note that the closure $\clE$
(taken in $Y$ or, equivalently, $\R^2$)
separates $Y$ and thus is not removable for
  $\hNp(Y \setm \clE)$.

Since $p>2$
it is well known that $\CpRtwo(\{x\}) =\CpRtwo(\{0\}) >0$ for $x \in \R^2$.
It is not difficult to see that
$\CpY(\{x\}) \ge \tfrac14 \CpRtwo(\{x\})$
for $x \in Y$.
Thus, by definition,
every \p-quasiopen set in $Y$ is open.
By Theorem~\ref{thm-BBM-gen} every \p-path open set in $Y$ is open,
and in particular $X$ is not \p-path open.
\end{example}  

By adding a weight, we now modify the previous example to 
cover all $p\ge1$.

\begin{example}
  Let $Y$, $E$ and $X$ be as in Example~\ref{ex-bow-tie-p>2},
    but this time we
equip $Y$ with the measure $d\mu=w\,dx$, where $w(x)=|x|^{-1}$, which is globally doubling on $Y$.
Then $Y$ supports a global $1$-Poincar\'e inequality, by
\cite[Example~A.24]{BBbook}.
The same proof also shows that $X$ supports a global $1$-Poincar\'e inequality.
By  Theorem~\ref{thm-path-almost-open-char-intro}, $X$ is \p-path almost
open in $Y$ for every $p \ge 1$.
Thus, by Theorem~\ref{thm-hN1-rem-global}, $E$ is isometrically removable
for $\hNp(X)$.
Note that $\clE$ separates $Y$ and thus is not removable for
$\hNp(Y \setm \clE)$.

We shall see that $X$ is not \p-path open in $Y$ for any $p \ge 1$.
This will be done by showing that
$E$ is not $1$-thin at $0=(0,0)$ and hence that $X$ is not $1$-finely open.
Since $\ConeY(\{0\})>0$, by \cite[Example~A.24 and Lemma~6.15]{BBbook}, it 
then follows from Theorem~\ref{thm-BBM-gen} that $X$ is not $1$-path open
in $Y$.

Let $u$ be a function admissible for $\coneY(E \cap B(0,b),B(0,2b))$
and let $g$ be an upper gradient of $u$.
Then
for each $0<a<b$,
\[
   \int_0^{2b} g(a,t) w(a,t)\,dt
   \ge \frac{1}{2\sqrt{2} b} \int_0^{2b} g(a,t) \,dt
   \ge \frac{1}{2\sqrt{2} b},
\]
since $g$ is an upper gradient, $u(a,0)=1$ and $u(a,2b)=0$.
It follows that
\[
   \int_{B(0,2b)} g w\,dx
   \ge \iintlim{0}{b} \int_0^{2b} g(a,t) w(a,t)\,dt\,da
   \ge b \frac{1}{2\sqrt{2} b}
   =  \frac{1}{2\sqrt{2}}.
\]
Hence, by taking infimum over all such $u$ and $g$, we see that
\[
  \coneY(E \cap B(0,b),B(0,2b)) \ge \frac{1}{2\sqrt{2}}.
\]
Testing with $u(x)=\min\{(2b-|x|)/b,1\}_\limplus$ shows
that
\[
  \coneY(B(0,b),B(0,2b)) \le \pi,
\]
 and so
  $E$ is not $1$-thin at $0$ by the definition~\eqref{deff-thin-p1},
  and thus $X$ is not $1$-path open in $Y$.
\end{example}

With a bit more work we can create similar examples of removable sets 
$E$ with non-\p-path open complements in unweighted $\R^n$. 
Moreover, it can be done so that any $E'\supset E$ with \p-path open complement
is not removable.

We start with the following result.
As in Example~\ref{ex-Rn-countable} this gives a lot of examples
of removable sets whose closure is not removable.

\begin{prop} \label{prop-Hnone}
Let $\Om \subset \R^n$, $n \ge 2$, be open and equipped with the Lebesgue measure $\Leb^{n}$.
Let  $E \subset \Om$ be a set with $(n-1)$-dimensional
Hausdorff measure $\Hnone(E)=0$.
Then $E$ is isometrically removable for $\hNp(\Om \setm E)$
for every $p \ge 1$.
\end{prop}

\begin{proof}
The proof is essentially identical to the proof of 
Theorem~\ref{thm-A}\,\ref{k-isom}.
However, this time we do not have any exceptional directions
as given by \eqref{eq-lines}.
\end{proof}  

\begin{example} \label{ex-hyperplane-1}
  Let $Y=\R^n$, $n \ge 3$, equipped with the Lebesgue measure $\Leb^n$,
  and $A \subset [0,1]$ be a nonempty set 
of zero 1-dimensional Lebesgue measure.
Let $\{q_j\}_{j=1}^{\infty}$ be an enumeration of $\Q$
and let $X = \R^n \setm E$, where
\begin{equation*}    \label{eq-def-E0}
E= \biggl( \bigcup_{i,j=1}^{\infty}(q_j+2^{-i} A) \biggr) \times
\R^{n-2} \times \{0\}.
\end{equation*}
Let $p>2-d$, where $0\le d\le 1$ is the Hausdorff dimension of $A$.
Note  that all $p>1$ are included when $\dim_H A=1$
  and $\Leb^1(A)=0$.
It follows from Proposition~\ref{prop-Hnone} that $E$ is removable
for $\hNp(X)$.
As $E$ is contained in the hyperplane $H:=\R^{n-1}\times\{0\}$, $X$ is a union
of the open set
$\R^n \setm H$ and a set of measure zero,
and thus \p-path almost open in $\R^n$,
by Theorem~\ref{thm-path-almost-open-char-intro}.
We shall now show that $X$ is not \p-path open in $\R^n$.

By Theorem~\ref{thm-BBM-gen}, this amounts to showing that 
$\CpRn(X \setm \fineint X)>0$,
where $\fineint X$ denotes the \p-fine interior of $X$, which consists
of all points $x\in X$ for which
\begin{equation}   \label{eq-Wien-ex}
  \sum_{i=0}^\infty \biggl(\frac{\cpRn(B(x,2^{-i}) \cap E,B(x,2^{1-i}))}
           {\cpRn(B(x,2^{-i}),B(x,2^{1-i}))}\biggr)^{1/(p-1)}  < \infty,
\end{equation}
see Mal\'y--Ziemer~\cite[Theorem~2.136]{MZ}.
We alert the reader that it is not enough to show that
$\CpRn(X \setm \interior X)>0$, since e.g.\ the complement of 
any countable dense set in $\R^n$, $n\ge p$, is \p-path open but 
has empty interior.

By Heinonen--Kilpel\"ainen--Martio~\cite[Lemma~12.10]{HeKiMa},
\eqref{eq-Wien-ex} is equivalent to the \p-thinness condition~\eqref{deff-thin}.
It is clear that \eqref{eq-Wien-ex} holds  for all $x\in X\setm H$.
For $x=(x_1,\ldots,x_n) \in X \cap H$ and $r=2^{1-i}$, $i=1,2,\ldots$\,,
find 
\[
y=(y_1,\ldots,y_n)\in H \quad \text{with} \quad y_1 \in \Q 
      \text{ and }|x-y|<\tfrac{1}{2}r = 2^{-i}.
\]
Let 
\[
A_i:= 2^{-i}A \times \R^{n-2} \times \{0\}, \quad i=0,1,\ldots.
\]
Then, by the scaling property and translation invariance of $\cpRn$
together with the construction of $E$,
\begin{align*}
\cpRn(B(x,r) \cap E,B(x,2r))
&\ge \cpRn(B(y,\tfrac12r ) \cap (y+ A_i),
B(y,\tfrac52r)) \\
&=  2^{-i(n-p)} \cpRn(A_0\cap B(0,1), B(0,5)) \\
&=: C_0 r^{n-p}.
\end{align*}
Since $A_0$ is $(d+n-2)$-dimensional and $p>2-d$, it follows from
e.g.\
Heinonen--Kilpel\"ainen--Martio~\cite[Theorem~2.26]{HeKiMa}
that $C_0>0$.
It is crucial here that $C_0$, by its definition above,
	only depends on the set $A$ fixed at the beginning,
	and not on the ball $B(x,r)$.
Similarly,
\[
  \cpRn(B(x,r),B(x,2r))=Cr^{n-p}
\quad \text{for some } C >0.
\]
Hence for all $r=2^{1-i}$, $i=1,2,\ldots$\,,
\[
\biggl(\frac{\cpRn(B(x,r) \cap E,B(x,2r))}
      {\cpRn(B(x,r),B(x,2r)} \biggr)^{1/(p-1)}\ge\frac{C_0}{C} >0,
\]
and inserting this into the Wiener criterion~\eqref{eq-Wien-ex}
shows that $x\notin \fineint X$, and hence $\fineint X=\R^n \setm H$.
Moreover, $H\setm E$ has infinite $(n-1)$-dimensional Hausdorff
measure and thus by e.g.\ \cite[Theorem~2.26]{HeKiMa} again,
\[
 \CpRn(X \setm \fineint X)=\CpRn(H\setm E)>0,
\]
i.e.\ $X$ is not \p-path open, by Theorem~\ref{thm-BBM-gen}.

It also follows that if $E' \supset E$ is any set
such that $X'=\R^n \setm E'$ is \p-path open (and
$\Leb^n(E')=0$),
then
$\CpRn(H \cap X')=0$. 
Since $H\cap X'$ separates $X'$, it follows that 
$X'$ cannot support
  a \p-Poincar\'e inequality and thus $E'$ is not removable
for $\hNp(X')$,
by Theorem~\ref{thm-hN1-rem-global}.
  Thus the removability of $E$ cannot be achieved by considering
  larger sets with \p-path open complements.
\end{example}

\section{\texorpdfstring{\p}{p}-path almost open sets}
\label{sec:p-path-almost-open}

Despite
the example given in Theorem~\ref{thm-A},
\p-path almost open sets
played a rather central
role in our studies of removable sets in Section~\ref{sect-rem}.
In this section, we therefore characterize \p-path almost open sets,
and in particular answer Open problem~3.4 in Bj\"orn--Bj\"orn~\cite{BBnonopen},
which asked whether every \p-path almost open set can be written as 
a union of a \p-path open set and a set of a measure zero.
We give an affirmative answer
for measurable sets, under natural assumptions.
At the same time, we also answer it in the negative for
nonmeasurable sets in unweighted $\R^n$, $n \ge 2$,
and give a measurable counterexample with a nondoubling underlying measure 
on $\R$.

We call a set $N\subset X$ \emph{\p-path negligible} if
for \p-almost every arc-length parameterized curve
$\gamma$ we have $\Leb^1(\gamma^{-1}(N))=0$, 
where $\Leb^1$ denotes the $1$-dimensional Lebesgue measure.
(Recall that we only consider rectifiable curves.)
A \p-path negligible set is obviously \p-path almost open.

It is easy to check that a set of measure zero is
\p-path negligible, see
Shan\-mu\-ga\-lin\-gam~\cite[Proof of Lemma~3.2]{Sh-rev}
  (or \cite[Lemma~1.42]{BBbook}).
Conversely, we have the following result.

\begin{prop}\label{prop-path-negligible}
  Assume that $\mu$ is locally doubling and supports a
 local \p-Poincar\'e inequality.
  Let $N\subset X$ be measurable and \p-path negligible.
  Then $\mu(N)=0$.
\end{prop}

Proposition~\ref{prop-exist-nonmeas-path-almost-open-set}
  below shows that the measurability assumption cannot be dropped.

\begin{proof}
  First we make the following observation: if $u\in N^{1,p}(X)$,
 then the minimal \p-weak upper gradient satisfies
  $g_u=0$ a.e.\ in $N$. To see this, note that
  for \p-almost every curve $\gamma$, we have $\Leb^1(\gamma^{-1}(N))=0$
  and so
  \[
  \int_{\gamma}g_u\,ds=\int_{\gamma}g_u\chi_{X\setminus N}\,ds.
  \]
  Thus $g_u\chi_{X\setminus N}$ is also a \p-weak upper gradient of $u$,
  and then by the minimality of $g_u$,
  we must have $g_u=0$ a.e.\ in $N$.
  
  In order to prove that $\mu(N)=0$, suppose instead that $\mu(N)>0$.
  Then there exists a point $x\in N$ of density one,
  see e.g.\ Heinonen~\cite[Theorem~1.8]{heinonen}.
  For each $i=1,2\ldots$\,, let $B_i=B(x,i^{-1})$ and
  $\eta_i(y)=(1-i\dist(y,B_i))_\limplus$.
Then
  $g_{\eta_i}\le  i \chi_{2B_i}$
  and in fact $g_{\eta_i}\le  i\chi_{2B_i\setminus N}$, by the earlier observation.
  
  By the local \p-Poincar\'e inequality,
  we have for all sufficiently large $i$ that the sphere
    $\bdy\tfrac52 B_i$ is nonempty and
\begin{equation} \label{eq-int-3Bi}
  \vint_{3B_i}
      |\eta_i - c_i| \,d\mu
      \le \frac{C}{i} \biggl( \vint_{3B_i}
      g_{\eta_i}^{p}
       \,d\mu \biggr)^{1/p},
\end{equation}
where $c_i:= \vint_{3B_i} \eta_i \,d\mu$ is the integral average.
Considering the cases $c_i\le\tfrac12$ and $c_i\ge\tfrac12$ separately,
we conclude that
the left-hand side satisfies
\[
  \vint_{3B_i}
  |\eta_i - c_i| \,d\mu
\ge \frac{\min\{\mu(B_i),\mu(3B_i\setm 2B_i)\}}{2\mu(3B_i)}
\ge \frac{1}{C'},
\]
by the local doubling property (and for large $i$).
On the other hand, the right-hand side satisfies
\[
  \frac{1}{i} \biggl( \vint_{3B_i} g_{\eta_i}^{p} \,d\mu \biggr)^{1/p}
  \le  \biggl( \frac{\mu(3B_i\setminus N)}{\mu(3B_i)} \biggr)^{1/p},
\]
which
tends
  to zero as $i\to\infty$, since $x$ is a density point of $N$.
  This contradicts
\eqref{eq-int-3Bi},
and so we have the result.
\end{proof}

Next we prove the following characterization of \p-path almost open sets.
Note that it applies also to nonmeasurable sets.

\begin{thm}        \label{thm-path-almost-open-char}
Assume that 
$X$ is locally compact and that $\mu$ is locally doubling
and supports a local \p-Poincar\'e inequality.
	Then $U \subset X$ is \p-path almost open  if and only if
	it can be written as a
        union
	$U=V\cup N$, where $V$ is \p-path open and $N$ is \p-path negligible.
\end{thm}

Recall that under these assumptions a set is \p-path open
if and only if it is \p-quasiopen, by 
Theorem~\ref{thm-BBM-gen}.

\begin{proof}
If $U=V\cup N$, where $V$ is \p-path open and $N$ is \p-path negligible,
then it is easy to see that $U$ is \p-path almost open.
	
Conversely, suppose that $U$ is \p-path almost open.
Now the family $\Gamma$ of curves $\gamma$, for which $\gamma^{-1}(U)$ is not the union
of an open set and a set of zero $1$-dimensional Lebesgue
measure,
has zero \p-modulus,
i.e.\ there is 
a Borel function
$0\le\rho\in L^p(X)$ such that $\int_{\gamma}\rho\,ds=\infty$
for every $\gamma\in\Gamma$, see \cite[Proposition~1.37]{BBbook}.
	
Assume first that $U$ is bounded and let $B$ be a ball containing
a $1$-neighbourhood of $U$.
Define
\[
u(x)=\min\biggl\{1,\inf_{\gamma}\int_{\gamma}(\rho+\chi_B)\,ds\biggr\},\quad x\in X,
\]
where the infimum is taken over all rectifiable curves (including constant curves)
from $x$ to $X\setminus U$.
Then
$u=0$ in $X\setminus U$, and $\rho+\chi_B$ is an upper gradient of $u$,
by Bj\"orn--Bj\"orn--Shanmugalingam~\cite[Lemma~3.1]{BBS5}
(or~\cite[Lemma~5.25]{BBbook}).
By Corollary~1.10 in
J\"arvenp\"a\"a--J\"arvenp\"a\"a--Rogovin--Rogovin--Shan\-mu\-ga\-lin\-gam~\cite{JJRRS}
(or Theorem~\ref{thm-JJRRS-gen}), $u$ is measurable.
As $u$ and $U$ are bounded and $\rho\in L^p(X)$, it
follows that $u\in N^{1,p}(X)$.
	
Let $V=\{x\in U : u(x)>0\}=\{x\in X : u(x)>0\}$ and 
$N=U \setm V$.
Then $V$
is \p-path open, since $u\in N^{1,p}(X)$ 
is (absolutely) continuous on \p-almost every curve in $X$,
by Proposition~3.1 in Shan\-mu\-ga\-lin\-gam~\cite{Sh-rev}
(or \cite[Theorem~1.56]{BBbook}).
It remains to show that $N$ is \p-path negligible.
Assume it is not. Then there necessarily
is an arc-length parameterized  curve $\widehat{\gamma}$ for which
$\Leb^{1}(D)>0$, where $D:=\widehat{\gamma}^{-1}(N)$,
but $\int_{\widehat{\gamma}}\rho\,ds<\infty$.

Let $x\in N$ and $0<\de\le 1$.
As $u(x)=0$, there are arc-length parameterized
curves $\gamma_j\colon [0,l_{\ga_j}] \to X$, $j=1,2,\ldots$\,,
such that $\gamma_j(0)=x$,
$\gamma_j(l_{\gamma_j})\in X\setminus U$
and
\[
\int_{\gamma_j}(\rho+\chi_B)\,ds\le 2^{-j-1}\de.
\]
Since $B$ contains a $1$-neighbourhood of $U$, necessarily
$l_{\gamma_j}\le 2^{-j-1}\de$.
We define a curve $\gamma_x$ as follows.
Let $L_0=0$ and for $i=1,2\ldots$\,,
\[
L_i=2\sum_{j=1}^i l_{\gamma_j}\le  2\sum_{j=1}^{\infty} l_{\gamma_j}
=: L\le 2\de\sum_{j=1}^{\infty} 2^{-j-1}=\de,
\]
and
\[
\gamma_x=
	\begin{cases}
	\gamma_j(t-L_{j-1}) &\textrm{for }L_{j-1}\le t\le L_{j-1}+l_{\gamma_j},\\
	\gamma_j(L_{j}-t) &\textrm{for }L_{j-1}+l_{\gamma_j}\le t\le L_{j},\\
	\end{cases}
\]
and $\gamma_x(L):=x$.
Then $\gamma_x\colon [0,L]\to X$
is an arc-length parameterized  curve with
$\gamma_x(0)=x=\gamma_x(L_j)=\gamma_x(L)$
and $\gamma_x(L_j+l_{\gamma_{j+1}})\in X\setminus U$
for all $j=1,2\ldots$\,, with $L_j+l_{\gamma_{j+1}}\to L$ as $j\to\infty$.
Also, $\length(\ga_x)= L\le\de$ and $\int_{\gamma_x}\rho\,ds\le \de$.
In essence, $\gamma_x$ is a short ``zigzagging loop'' at $x$
which intersects $X\setminus U$ arbitrarily close to its end point.
	
Now take a dense set $\{s_k\}_{k=1}^{\infty}\subset D$. 
For every $k=1,2,\ldots$\,, we find such a zigzagging loop
$\gah_k:=\ga_{x_k}$ at $x_k=\widehat{\gamma}(s_k)$, with
$l_{\gah_k}\le 2^{-k}$ and $\int_{\gah_k}\rho\,ds\le 2^{-k}$.
Next we define a curve $\gamma$ that is obtained
from $\widehat{\gamma}$ by adding the ``loops'' $\gah_k$ at the points
$x_k$, for $k=1,2,\ldots$\,.
More precisely, first let $l=\sum_{k=1}^{\infty}l_{\gah_k}$.
Then define the function
\[
f\colon [0,l_{\widehat{\gamma}}]\to [0,l_{\widehat{\gamma}}+l],\quad
	f(t):=\nu([0,t])\ \ \textrm{with}\ \
	\nu=\Leb^1+\sum_{k=1}^{\infty}l_{\gah_k}\delta_{s_k},
\]
where $\delta_{s_k}$ are Dirac measures at the points $s_k$.
Now $f^{-1}$ is defined on a subset of $[0,l_{\widehat{\gamma}}+l]$ and is $1$-Lipschitz.
We define a curve $\gamma$ on $[0,l_{\widehat{\gamma}}+l]$ as follows.
For $t\in f([0,l_{\widehat{\gamma}}])$, let
$\gamma(t)=\widehat{\gamma}(f^{-1}(t))$.
If $t\in [0,l_{\widehat{\gamma}}+l]\setminus f([0,l_{\widehat{\gamma}}])$,
then for some $k=1,2\ldots$\,, the number $t$ belongs to an interval
of length $l_{\gah_k}$ which does not intersect
$f([0,l_{\widehat{\gamma}}])$ apart from the right end point $f(s_k)$.
Define $\gamma$ to be the curve $\gah_k$ on this interval.
Note that $\gamma$ is a $1$-Lipschitz mapping
and that $\length(\ga)
  =l_{\widehat{\gamma}}+l$.
Thus $\gamma$ is arc-length parameterized, and so it is indeed a ``curve''
in our sense.
	
Since $\widehat{\gamma}(D)\subset N$, we also get $\gamma (f(D))\subset N$.
Moreover, since $f^{-1}$ is $1$-Lipschitz,
$\Leb^1(f(D))\ge \Leb^1(D)>0$ and so
$\gamma$ travels a positive length in $N$.
Let $t:=f(\xi)\in f(D)$ and $\eps>0$.
Then by the construction of $f$, together with the density of
$\{s_k\}_{k=1}^{\infty}$ in $D$, we can find $k$ and $j_0(k)$
such that
$\lim_{k \to \infty} j_0(k)=\infty$ and
\[
|f(s_k)-t|=
|f(s_k)-f(\xi)| \le |s_k-\xi| + \sum_{j\ge j_0(k)} l_{\gah_j}
\le |s_k-\xi| +2^{1-j_0(k)} <
\eps.
\]
By the construction of the zigzagging loop
$\gah_k$, there is
a sequence $t_l\nearrow f(s_k)$
such that $\gamma(t_l)\in X\setminus U$ for
$l=1,2,\ldots$\,.
Since $\eps>0$ was arbitrary, we conclude that
$t$ is not in the interior of $\gamma^{-1}(U)$.
Thus no $t\in f(D)$ is an interior point
of $\gamma^{-1}(U)$, and since we had $\Leb^1(f(D))>0$,
$\gamma^{-1}(U)$ is not the union of a relatively open set
and a set of zero $\Leb^1$-measure.
This shows that $\ga\in\Ga$.
	
At the same time,
\[
\int_{\gamma}\rho\,ds =\int_{\widehat{\gamma}}\rho\,ds
	+\sum_{k=1}^{\infty}\int_{\gah_k}\rho\,ds
\le \int_{\widehat{\gamma}}\rho\,ds+\sum_{k=1}^{\infty}2^{-k}
=
\int_{\widehat{\gamma}}\rho\,ds+1<\infty.
\]
This contradicts the choice of $\rho$.
Thus $N$ is in fact a \p-path negligible set and we have the result
for bounded sets $U$.
	
If $U$ is \p-path almost open and unbounded, we know that each
$U\cap B(x_0,j)$
is a disjoint union of a \p-path open set $V_j$ and a \p-path negligible
set $N_j$, $j=1,2\ldots$\,, where
$x_0 \in X$ is fixed.
Now we can write $U$ as the union
\[
     U=\bigcup_{j=1}^{\infty}V_j\cup
    \bigcup_{j=1}^{\infty}N_j,
\]
where $\bigcup_{j=1}^{\infty}V_j$ is obviously \p-path open
and $\bigcup_{j=1}^{\infty}N_j$
is \p-path negligible.
\end{proof}

Finally, we obtain the following natural
characterization of measurable \p-path almost open sets.
This answers Open problem~3.4 in Bj\"orn--Bj\"orn~\cite{BBnonopen}
  in the affirmative for measurable sets, under natural assumptions.

\begin{thm} \label{thm-p-path-almost-open-char}
Assume that 
$X$ is locally compact and that $\mu$ is locally doubling
and supports a local \p-Poincar\'e inequality.
  Suppose that $U\subset X$ is measurable.
  Then $U \subset X$ is \p-path almost open  if and only if
  it can be written as 
  $U=V\cup N$, where $V$ is \p-path open and $\mu(N)=0$.
\end{thm}

Under these assumptions, it follows
from Theorem~\ref{thm-BBM-gen}
  that every \p-path open set is \p-quasiopen and thus measurable.
  Hence it follows from 
  Proposition~\ref{prop-exist-nonmeas-path-almost-open-set}
  below that the measurability assumption
  in Theorem~\ref{thm-p-path-almost-open-char}
   cannot be dropped.

\begin{proof}
  If $U$ is \p-path almost open, then by
  Theorem~\ref{thm-path-almost-open-char} we know that
  it is a
  union
  $U=V\cup N'$ where $V$ is \p-path open and $N'$ is \p-path
  negligible.
Then $U=V \cup N$, where $N=N' \setm V$ is also \p-path negligible.
By Theorem~\ref{thm-BBM-gen}, $V$ is measurable.
As
$U$ is measurable by
  assumption, so is
$N=U\setminus V$.
  Thus by Proposition~\ref{prop-path-negligible}, we have $\mu(N)=0$.
  
  Conversely, if $U=V\cup N$, where $V$ is \p-path open
  and $\mu(N)=0$, then $N$ is \p-path negligible by
  \cite[Lemma~1.42]{BBbook}, and hence
  $U$ is \p-path almost open
  by Theorem~\ref{thm-path-almost-open-char}.
\end{proof}

A natural question is whether there exist nonmeasurable
\p-path almost open sets.
If there are no nonconstant rectifiable curves in $X$, as e.g.\
on the von Koch snowflake curve, then
all sets are \p-path open as well
as \p-path almost open, and thus there are plenty of
nonmeasurable
\p-path open and \p-path almost open sets.
But what can happen under natural assumptions, 
such as doubling and a Poincar\'e inequality?

First consider the $1$-dimensional case.

\begin{prop}
Let $X=\R$ be equipped with a locally doubling measure $\mu$ 
supporting a local \p-Poincar\'e inequality.
Then every \p-path almost open set $G$ is a union
of an open set and a set of measure zero, and is in particular
measurable.
\end{prop}

\begin{proof}
By Bj\"orn--Bj\"orn--Shanmugalingam~\cite[Theorem~1.2]{BBSpadm},
$d\mu=w\,dx$ and $w$ is a local $A_p$-weight.
Let $a>0$ and $\ga\colon [-a,a] \to \R$ with $\ga(t)=t$.
If $\rho\ge0$ is a function admissible in the definition of
$\Mod_{p,X}(\{\ga\})$ and $p>1$, then
\[
1\le \int_{\ga} \rho\,ds 
   = \int_{-a}^a \rho w^{1/p}w^{-1/p}\,dx 
   \le \biggl( \int_{-a}^a \rho^p \,d\mu\biggr)^{1/p}
   \biggl( \int_{-a}^a
   w^{-1/(p-1)}\,dx\biggr)^{(p-1)/p}.
\]
Taking infimum over all such $\rho$ and in view of   
the local $A_p$-condition \cite[(5.1)]{BBSpadm}, we see that
the single curve family $\{{\ga}\}$ 
has positive \p-modulus.
(The calculation is similar when $p=1$.)
Thus
necessarily $\ga^{-1}(G)=G \cap [-a,a]$
is a 
union of an
open set and 
a set of measure zero. 
Hence also $G=\bigcup_{k=1}^\infty (G \cap [-k,k])$
is a 
union of an open set and a set of measure zero.
\end{proof}

The same argument applies to any connected metric graph $X$
equipped with a locally
doubling measure $\mu$ supporting a local \p-Poincar\'e inequality, 
where each edge
is considered to be a 
segment.
To see this, first note that there are at most a countable number
of vertices and edges, and that $\mu(\{x\})=0$ for each $x \in X$,
see \cite[Corollary~3.9]{BBbook}.
It follows that the set of vertices has zero measure.
On each open edge, $\mu$
is given by a locally \p-admissible weight,
by \cite[Theorem~4.6]{BBSpadm},
and we can apply the argument above.

On the contrary, in higher dimensions there always
exist nonmeasurable \p-path almost open sets,
at least if we assume the continuum hypothesis.

\begin{prop} \label{prop-exist-nonmeas-path-almost-open-set}
Assume that the continuum hypothesis is true.
Let $X=\R^n$, $n \ge 2$, be equipped with a measure $d\mu=w\,dx$
such that $0 < w \in L^1\loc(\R^n)$.

Then there is a nonmeasurable dense \p-path negligible set $S$.
In particular, $S$ is a nonmeasurable dense \p-path almost open set.
\end{prop}

In particular,
Proposition~\ref{prop-exist-nonmeas-path-almost-open-set}
applies to \p-admissible weights $w$,
as studied extensively in
Heinonen--Kilpel\"ainen--Martio~\cite{HeKiMa} when $p>1$.
Note that $\mu$ and the Lebesgue measure $\Leb^n$ have the same measurable sets.

We shall use Sierpi\'nski sets to prove
Proposition~\ref{prop-exist-nonmeas-path-almost-open-set}.
A \emph{Sierpi\'nski set} $S$ is an uncountable subset of $\R^n$
such that $E \cap S$ is at most countable
for every set $E$ of
Lebesgue measure $\Leb^n(E)=0$.
Such sets exist if we assume the continuum hypothesis,
see Sierpi\'nski~\cite{sierpinski}
(Proposition~$C_{26}$ in~\cite[p.~80]{sierpinski} gives the existence
for $\R$, while in the paragraph just before 
Proposition~$C_{26^{\scriptstyle a}}$
in~\cite[p.~81]{sierpinski} it is explained
how to deduce the existence for $\R^2$) and Morgan~\cite[Theorem~7, p.~86]{Morgan}
(for $\R^n$).
On the other hand, there
are other models of set theory
containing ZFC (Zermelo--Fraenkel's system plus the axiom of choice)
for which the existence of Sierpi\'nski sets fails,
e.g.\ if one adds Martin's axiom for $\aleph_1$,
see Kunen~\cite[Exercise~V.6.29]{kunen}.

Let $S \subset \R^n$, $n \ge 2$, be a Sierpi\'nski set
and 
$A \subset S$.
Then $A \cap H \subset S \cap H$ is at most countable for every hyperplane $H$.
If $A$ is 
measurable, then it 
follows from Fubini's theorem that $\Leb^n(A)=0$,
but then $A =A \cap S$ is  
at most countable. 
Thus every uncountable subset of $S$ is nonmeasurable. In
particular $S$  itself is nonmeasurable.
Conversely it is
easy to show that if
$S \subset \R^n$, $n \ge 1$, is an uncountable set 
such that every uncountable subset is nonmeasurable, then $S$ is
a Sierpi\'nski set.

In fact, there exist Sierpi\'nski sets with additional, perhaps
surprising, properties. 
For example,
Bienias--G\l \aob b--Ra\l owski--\.Zeberski~\cite[Theorem~5.5]{BienasGRZ}
have shown that in $\R^2$ there is a Sierpi\'nski set that 
intersects every line in at most two points.
(This is again assuming the continuum hypothesis.)

When proving Proposition~\ref{prop-exist-nonmeas-path-almost-open-set}
  we will need the following lemma, which is no doubt well known.
  As we have not found a good reference, we provide a short proof.

\begin{lem}  \label{lem-zero-preim-x}
Let  $\ga\colon [0,l_\ga] \to X$ be an arc-length parameterized curve.
Then
\[
   \Leb^1(\ga^{-1}(x))=0
   \quad \text{for every }x\in X.
\]
\end{lem}

 \begin{proof} 
The metric derivative
  	\[
  	|\dot{\gamma}|(t):=\lim_{h\to 0}\frac{d(\gamma(t+h),\gamma(t))}{|h|}
  	\]
 satisfies $|\dot{\gamma}|(t)=1$ for a.e.\ $t\in [0,l_{\gamma}]$, 
 see e.g.\ Haj\l{}asz~\cite[Corollary~3.7]{Haj03}.
 At the same time, clearly $|\dot{\gamma}|(t)=0$ at every point $t$ 
 of density one for the closed set $\ga^{-1}(\{x\})$ (provided that the limit exists),
 and thus at a.e.\ $t\in \ga^{-1}(\{x\})$.
\end{proof}

\begin{proof}[Proof of Proposition~\ref{prop-exist-nonmeas-path-almost-open-set}]
By the assumptions on the measure $\mu$, it has the
same zero sets
and the same measurable sets
as the Lebesgue measure $\Leb^n$.
As mentioned above,
there exists a  Sierpi\'nski set $S'\subset\R^n$.
It is easy to see that a countable union of Sierpi\'nski sets is a
  Sierpi\'nski set, and hence $S=\bigcup_{q \in \Q^n} (S'+q)$ is a
  dense Sierpi\'nski set.

If  $\ga\colon [0,l_\ga] \to \R^n$ is an arc-length parameterized
curve,
then $\ga([0,l_\ga]) \cap S$ is at most countable, since
$\Leb^n(\ga([0,l_\ga]))=0$.
Lemma~\ref{lem-zero-preim-x} and the countable additivity of the
Lebesgue measure $\Leb^1$
then imply that $\Leb^1(\ga^{-1}(S))=0$.
As this holds for every curve $\ga$,
the set $S$ is \p-path negligible for every $p$.
However, $S$ is nonmeasurable with respect to $\Leb^n$, and
thus
also with respect to $\mu$.
\end{proof}

We end the paper by constructing a measurable \p-path almost
open set which cannot be written as a union of an open set and a set of measure zero.
Note that the measure is not doubling and does not support a Poincar\'e inequality.

\begin{example} \label{ex-Leb+Dirac}
Let $X=\R$, equipped with the measure $\Leb^1+\de_0$,
where $\de_0$ is the Dirac measure at $0$.
Then $C_p(\{x\})\ge2$ for all $x\in X$ and hence all
quasiopen sets in $X$ are open.
The interval $[0,1)$ cannot therefore be written
as a
union of a quasiopen set and a set of measure zero.
However, it is still \p-path almost open for any 
$p\ge1$, by Lemma~\ref{lem-zero-preim-x}.

For an example with a nonatomic measure, equip  
$\R\times(\R\setm\Q)^{n-1}$ with the measure $(\Leb^1+\de_0)\times\Leb^{n-1}$
and consider $U=[0,1) \times ((0,1)\setm\Q)^{n-1}$, $n\ge2$.
\end{example}

\end{document}